\documentclass{amsart}

\usepackage{amsaddr}
\usepackage{orcidlink}
\usepackage{appendix}
\usepackage{amstext}
\usepackage{amsthm}
\usepackage{amssymb}
\usepackage{enumerate}
\usepackage{color}

\newtheorem{theorem}{Theorem}[section]
\newtheorem*{theorem*}{Theorem}
\newtheorem{proposition}[theorem]{Proposition}
\newtheorem{lemma}[theorem]{Lemma}
\newtheorem{corollary}[theorem]{Corollary}
\newtheorem{fact}[theorem]{Fact}
\theoremstyle{definition}
\newtheorem{definition}[theorem]{Definition}

\newcommand{\M}{\mathcal{M}}
\newcommand{\e}{\varepsilon}

\begin{document}
\title{Kuroda's theorem for $n$-tuples in semifinite von Neumann algebras}

\author[A.\ Ber]{Aleksey Ber}
\address{V.I.Romanovskiy Institute of Mathematics Uzbekistan Academy of Sciences
    University street, 9, Olmazor district, Tashkent, 100174, Uzbekistan,
    National University of Uzbekistan
    4, Olmazor district, Tashkent, 100174, Uzbekistan}
\email{aber1960@mail.ru}

\author[F.\ Sukochev]{Fedor Sukochev}
\address{School of Mathematics and Statistics, University of New South Wales, Kensington, NSW 2052, Australia}
\email{f.sukochev@unsw.edu.au}

\author[D.\ Zanin]{Dmitriy Zanin}
\address{School of Mathematics and Statistics, University of New South Wales, Kensington, NSW 2052, Australia}
\email{d.zanin@unsw.edu.au}

\author[H.\ Zhao]{Hongyin Zhao \orcidlink{0000-0003-2591-475X}}
\address{School of Mathematics and Statistics, University of New South Wales, Kensington, NSW 2052, Australia}
\email{hongyin.zhao@unsw.edu.au}

\maketitle
\begin{abstract} Let $\M$ be a semifinite von Neumann algebra and let $E$ be a symmetric function space on $(0,\infty)$. Denote by $E(\M)$ the non-commutative symmetric space of measurable operators affiliated with $\M$ and associated with $E.$ Suppose $n\in \mathbb{N}$ and $E\cap L_{\infty}\not\subset L_{n,1}$, where $L_{n,1}$ is the Lorentz function space with the fundamental function $\varphi(t)=t^{1/n}$. We prove that for every $\varepsilon>0$ and every commuting self-adjoint $n$-tuple $(\alpha(j))_{j=1}^n,$ where $\alpha(j)$ is affiliated with $\M$ for each $1\leq j\leq n,$ there exists a commuting $n$-tuple $(\delta(j))_{j=1}^n$ of diagonal operators affiliated with $\M$ such that $\max\{\|\alpha(j)-\delta(j)\|_{E(\M)},\|\alpha(j)-\delta(j)\|_{\infty}\}<\e$ for each $1\le j\le n$. In the special case when $\M=B(H)$, our results yield the classical Kuroda and Bercovici-Voiculescu theorems.
\end{abstract}

\section{Introduction}\label{intro sec}

Let $H$ be a (possibly non-separable) Hilbert space. Let $B(H)$ be the $\ast$-algebra of all bounded operators on $H$ with the uniform norm $\|\cdot\|_{\infty}$. Recall that an operator $d$ on $H$ is called {\em diagonal} if there exists a sequence of pairwise orthogonal projections $\{p_k\}_{k\in \mathbb{N}}$ in $B(H)$ such that $\sum_{k\ge1}p_k={\bf 1}$, and $dp_k=\lambda_kp_k$ where $\lambda_k\in \mathbb{C},$ for each $k\in \mathbb{N}.$ A commuting $n$-tuple of operators $(\delta(j))_{j=1}^n$ on $H$ is called {\em diagonal} if $\delta(j)$ is diagonal for each $1\le j\le n.$ 

In what follows, the symbol $H_0$ always denotes a separable Hilbert space. Kuroda's Theorem is one of the fundamental results in perturbation theory (see \cite{Kuroda1958} and \cite[Theorem X.2.3]{Kato1995}). It states that, if $\mathcal{J}$ is a Banach ideal (or in the term used in \cite{Gohberg1969}, symmetrically normed ideal) in $B(H_0)$ that is not contained in the trace class, then for every self-adjoint operator $b$ on $H_0$ and every $\e>0,$ there exists a diagonal operator $d$ on $H_0$ such that $\|b-d\|_{\mathcal{J}}<\e.$


In 1989, Bercovici and Voiculescu \cite{Bercovici1989} extended Kuroda's theorem to $n$-tuples of commutative self-adjoint operators, as follows: if a Banach ideal $\mathcal{J}\not\subset L_{n,1}(B(H_0))$, where $L_{n,1}(B(H_0))$  is the Lorentz-$(n,1)$ ideal in $B(H_0)$, then for every commuting self-adjoint $n$-tuple $(\alpha(j))_{j=1}^n\in (B(H_0))^n$ and every $\e>0$, there exists a diagonal $n$-tuple $(\delta(j))_{j=1}^n\in (B(H_0))^n$ such that 
$$\max_{1\le j\le n}\|\alpha(j)-\delta(j)\|_{\mathcal{J}}<\e.$$
Their extension was based on a method introduced by Voiculescu in \cite{V1979} for the treatment of the diagonality modulo Banach ideals. Voiculescu proved that the diagonality of $\alpha$ modulo $\mathcal{J}$ is equivalent to the existence of a quasicentral approximate unit for $\alpha$ relative to $\mathcal{J}$ (see \cite{V1979}).

Li et al. \cite{LSS2020} obtained a version of Kuroda's theorem for semifinite von Neumann algebras stating that if $\M$ is a properly infinite $\sigma$-finite semifinite von Neumann algebra $\M$ and $b\in\M$ is self-adjoint, then for every symmetric operator space (see \cite{LSS2020}) $\mathcal{J}\not\subset L_1(\M)$ and every $\e>0,$ there exists a diagonal operator $d\in\M$ such that $\|b-d\|_{\mathcal{J}}<\e.$ Here, $L_1(\M)$ is the non-commutative $L_1$ space (see \cite{PX2003}).

We now briefly introduce the semifinite setting and the corresponding notations; precise definitions will be provided in Section \ref{sec_pre}. Let $\M$ be a $\sigma$-finite von Neumann with a faithful normal semifinite trace $\tau.$ Let ${\rm Aff}(\M)$ be the collection of all linear operators $x$ on $H$ affiliated with $\M$. An operator $d\in {\rm Aff}(\M)$ is called {\em diagonal} if there exists a sequence of pairwise orthogonal projections $\{p_k\}_{k\in \mathbb{N}}$ in $\M$ such that $\sum_{k\in \mathbb{N}}p_k={\bf 1}$, and $dp_k=\lambda_kp_k$ where $\lambda_k\in \mathbb{C}$, for each $k\in \mathbb{N}.$ A commuting $n$-tuple $\delta=(\delta(j))_{j=1}^n\in ({\rm Aff}(\M))^n$ is called {\em diagonal} if $\delta(j)$ is diagonal for each $1\le j\le n.$

Suppose $(E,\|\cdot\|_E)$ is a symmetric function space on $(0,\infty)$ and $(E(\M),\|\cdot\|_{E(\M)})$ is the corresponding symmetric space associated with $\M$. We use the symbol $\alpha=(\alpha(j))_{j=1}^n$ to denote an $n$-tuple of operators. For two $n$-tuples $\alpha,\beta\in({\rm Aff}(\M))^n$, we write $\alpha+\beta=(\alpha(j)+\beta(j))_{j=1}^n.$ For every $n$-tuple $\gamma\in (E(\M))^n$, we write
\[\|\gamma\|_{E(\M)}=\max_{1\le j\le n}\|\gamma(j)\|_{E(\M)}.\]
We denote by $E^0$ the closure of $L_1\cap L_\infty$ in $E.$ The norm in $E\cap L_\infty$ is defined by
$$\|x\|_{E\cap L_\infty}=\max\{\|x\|_{E},\|x\|_{L_\infty}\},\quad x\in E\cap L_\infty.$$
Let $L_{n,1}$ be the standard Lorentz-$(n,1)$ space on $(0,\infty)$ (see e.g. \cite[p. 216]{Bennett1988}).

\begin{definition}\label{quasi approx unit} Let $E$ be a symmetric function space on $(0,\infty)$. A given commuting self-adjoint $n$-tuple $(\alpha(j))_{j=1}^n\in ({\rm Aff}(\M))^n$ is said to be \emph{diagonal modulo $E(\M)$}, if there exists a diagonal $n$-tuple $(\delta(j))_{j=1}^n\in ({\rm Aff}(\M))^n,$ such that $\alpha(j)-\delta(j)\in E^0(\M)$ for each $1\le j\le n$.
\end{definition}

The following theorem is the main result of the paper. It is a semifinite version of Kuroda-Bercovici-Voiculescu's theorem for $n$-tuples of operators. It extends \cite[Theorem 4]{Bercovici1989} to $\sigma$-finite semifinite von Neumann algebras. 

\begin{theorem}\label{thm_main} Let $\M$ be a $\sigma$-finite von Neumann algebra with a faithful normal semifinite trace $\tau$ and let $n\in \mathbb{N}$. Let $(E,\|\cdot\|_E)$ be a symmetric function space on $(0,\infty)$. If $E\cap L_\infty\not\subset L_{n,1}$, then for every commuting self-adjoint $n$-tuple $\alpha\in ({\rm Aff}(\M))^n$ and every $\e>0$, there exists a commuting diagonal $n$-tuple $\delta\in ({\rm Aff}(\M))^n$ such that
$$\alpha-\delta\in (E^0(\M)\cap \M)^n,\quad \|\alpha-\delta\|_{(E\cap L_\infty)(\M)}<\e.$$
\end{theorem}

In the special case when $\M=B(H_0)$ where $H_0$ is a separable Hilbert space, Theorem \ref{thm_main} recovers the necessity direction of \cite[Theorem 4]{Bercovici1989}. In the special case when $\M$ is properly infinite, $n=1$, and under an additional assumption that $\alpha$ is bounded, Theorem \ref{thm_main} recovers \cite[Theorem 3.2.2]{LSS2020}.

Our method is entirely different to that of Li et al. \cite{LSS2020}. Instead, it is inspired by the approach to double operator integration theory in semifinite von Neumann algebras setting pioneered by de Pagter, the second co-author and Witvliet in \cite{PSW2002} and whose origin, in the classical setting $\M=B(H_0)$, may be found in the outstanding paper by Birman and Solomyak \cite{Birman1973}. This approach allows us to reduce the proof of Theorem \ref{thm_main} to the case where $\alpha$ admits a generating $\tau$-finite projection (see Section \ref{construction sec}). In this case, we prove the existence of a quasicentral approximate unit $\{p_m\}_{m\ge1}$ for $\alpha$ (see Definition \ref{approx unit}) with certain good properties, see Lemma \ref{proj co}. In the case $n=1$, $\{p_m\}_{m\ge1}$ satisfies $\|[p_m,\alpha]\|_{E(\M)}\to0$ as $m\to\infty$, whenever $E$ is a symmetric function space on $(0,\infty)$ such that $E\not\subset L_1;$ see the proof of Lemma \ref{wvn pre lemma}. This recovers \cite[Lemma 6.3]{PSW2002}.

To prove Theorem \ref{thm_main}, we first obtain an intermediate result for Lorentz ideals, which is given in Section \ref{lorentz sec} (see Theorem \ref{kuroda thm lorentz}). Here, we need our recent result in \cite{BSZZ2024} concerning the extension of Voiculescu's results \cite{V1979,V2018} in semifinite von Neumann algebras. 
This result states that for a given symmetric function space $E$ on $(0,\infty)$, a commuting self-adjoint $n$-tuple $\alpha\in\M^n$ is diagonal modulo $E^0(\M)$ if and only if $k_{E(\M)}(\alpha)=0,$ see Theorem \ref{thm_1.2_BSZZ2024}. Here, $k_{E(\M)}(\alpha)$ is a number called the {\em quasicentral modulus}, introduced by Voiculescu \cite{V1979,V2021,V2022}; see Definition \ref{def_qc}. It quantifies the obstruction of the commuting self-adjoint $n$-tuple $\alpha$ modulo the symmetric operator space $E(\M).$

Finally, to generalise Kuroda's theorem for general symmetric function spaces, we employ the approach from Voiculescu \cite{V1990}. The proof of Theorem \ref{thm_main} is presented in Section \ref{symetric space sec}.

In the case when $n=1$, Theorem \ref{thm_main} yields the corollary below. We decided to single this case out, since we provide an (alternative) simpler proof for this special case (see Appendix \ref{app single operator}). This corollary yields a proper extension of \cite[Theorem 3.2.2]{LSS2020} to $\sigma$-finite semifinite von Neumann algebras, since the condition $E\not\subset L_1$ is equivalent to the assumption  used in \cite[Theorem 3.2.2]{LSS2020} that $t^{-1}\varphi_E(t)\to0$ as $t\to\infty$, see Lemma \ref{lemma in SS} below. Here $\varphi_E$ is the fundamental function of $E$, i.e. $\varphi_E(t)=\|\chi_{(0,t)}\|_E,t\ge0$ (see \cite[Section 2.4]{Krein1982}).

\begin{corollary}\label{cor_main} Let $\mathcal{M}$ be a $\sigma$-finite von Neumann algebra with a faithful normal semifinite trace $\tau$. Suppose $E$ is a symmetric function space on $(0,\infty)$ such that $E\not\subset L_1.$ Let $a\in {\rm Aff}(\M)$ be self-adjoint. For every $\varepsilon>0,$ there exists a diagonal operator $d\in {\rm Aff}(\M)$ such that
$$a-d\in E^{0}(\M)\cap \M,\quad \|a-d\|_{(E\cap L_\infty)(\M)}<\e.$$
\end{corollary}


In the special case when $\M=B(H_0)$ where $H_0$ is a separable Hilbert space, Corollary \ref{cor_main} is exactly the classical Kuroda's theorem \cite{Kuroda1958}. Indeed, in this case it suffices to consider the spaces $E$ which satisfy $E\subset L_\infty$. Assuming $E\subset L_\infty$, we have $E\subset L_1$ if and only if $E(B(H_0))\subset L_1(B(H_0))$. Hence, Corollary \ref{cor_main} retrieves the classical Kuroda's theorem \cite{Kuroda1958}.

A symmetric function space $E$ on $(0,\infty)$ is called an {\em $(\M,n)$-diagonalisation space} if every commuting self-adjoint $n$-tuple $\alpha\in ({\rm Aff}(\M))^n$ is diagonal modulo $E^0(\M).$ It is called an {\em $(\M,n)$-obstruction space} if there exists a commuting self-adjoint $n$-tuple $\alpha\in ({\rm Aff}(\M))^n$ that is not diagonal modulo $E^0(\M).$ 

Theorem \ref{thm_main} amounts to saying that if $E\cap L_\infty\not\subset L_{n,1}$, then $E$ is an $(\M,n)$-diagonalisation space. In the case where $\M=B(H_0)$, surprisingly, the converse statement is also true, as shown in \cite[Theorem 4]{Bercovici1989}. That is, if $E\cap L_\infty\subset L_{n,1}$, then $E$ is a $(B(H_0),n)$-obstruction space. This follows from an analogue of the classical Kato-Rosenblum theorem \cite{Kato1957,Rosenblum1957}, as presented by Voiculescu \cite{V1979,V1981}.

In Section \ref{sec_rem}, we present further remarks concerning the converse of Theorem \ref{thm_main} and Corollary \ref{cor_main}, particularly their connections to the Kato-Rosenblum theorem.


\section{Preliminaries}\label{sec_pre}

In this section, we recall notions of semifinite von Neumann algebras, symmetric spaces associated with a given semifinite von Neumann algebra, and quasicentral modulus of a given $n$-tuple in $\M.$

\subsection{$\tau$-measurable operators}\label{measur subsec}
Let $H$ be a (possibly non-separable) Hilbert space and let $\mathbf{1}$ be the identity operator on $H.$ Let $B(H)$ be the $\ast$-algebra of all bounded linear operators on $H$, equipped with the uniform operator norm $\|\cdot\|_\infty.$

For a closed densely defined operator $x$ on $H$, we denote by $\mathfrak{l}(x)$ the projection onto the subspace $\overline{x({\rm dom}(x))}$, which is called the \textit{left support} of $x$, i.e.\ $\mathfrak{l}(x)$ is the smallest projection $e\in B(H)$ for which $ex=x$. We denote by $\mathfrak{r}(x)$ the projection onto the subspace $(\ker x)^\perp$, which is called the \textit{right support} of $x$, i.e.\ $\mathfrak{r}(x)$ is the smallest projection $e\in B(H)$ for which $xe=x$.



%


A \emph{von Neumann algebra} $\M$ is a $\ast$-subalgebra of $B(H)$ such that $\M=\M''$, where $\M''$ is the double commutant of $\M.$ A von Neumann algebra $\M$ is called {\em semifinite} if there exists a faithful normal semifinite trace $\tau$ on $\M$ (see \cite{T1}). A von Neumann algebra $\M$ is \emph{$\sigma$-finite} if each orthogonal family of non-zero projections in $\M$ is countable.

Let $\mathcal M$ be a von Neumann algebra acting on $H$ with the uniform norm $\|\cdot\|_\infty$ and a faithful normal semifinite trace $\tau$. By $\mathcal P(\mathcal M)$ we denote the set of all projections in $\mathcal M.$ Let $\mathcal{F}(\M)=\{x\in\M:\tau(\mathfrak{l}(x))<\infty\}.$

Suppose $x$ is a self-adjoint operator on $H$. The spectrum of $x$ is denoted by ${\rm Spec}(x).$ The spectral measure of $x$ is denoted by $e^x:\mathfrak{B}(\mathbb{R})\to \mathcal{P}(B(H)),$ where $\mathfrak{B}(\mathbb{R})$ represents the collection of all Borel sets in $\mathbb{R}.$ 

\begin{definition}\label{commuting definition 1}
Let $a,b$ be self-adjoint operators on $H$. We say that $a$ and $b$ {\em commute} if their spectral measures commute (see \cite[p. 150]{Birman1987}).
\end{definition}

A linear operator $a:{\rm dom}(a)\to H$ is said to be {\em affiliated with $\M$} if $ba\subset ab$ for all $b$ from the commutant $\M'$ of $\M,$ and that the collection of all operators affiliated with $\M$ is denoted by ${\rm Aff}(\M).$

Let $n\in \mathbb{N}$. An $n$-tuple $\alpha$ of unbounded operators on $H$ is said to be affiliated with $\M$ if $\alpha(j)\in{\rm Aff}(\M)$ for every $1\le j\leq n.$
Similarly, $\alpha$ is said to be self-adjoint if $\alpha(j)$ is self-adjoint for every $1\le j\le n.$

Let $n\in \mathbb{N}$ and let $\alpha$ be a commuting self-adjoint $n$-tuple affiliated with $\M.$ Denote by $\mathfrak{B}(\mathbb{R}^n)$ the collection of all Borel sets in $\mathbb{R}^n$. The spectral measure of $\alpha$ is denoted by $e^{\alpha}:\mathfrak{B}(\mathbb{R}^n)\to \mathcal{P}(\M)$.

A closed densely defined operator $x:{\rm dom}(x)\to H$ affiliated with~$\mathcal{M}$ is said to be \emph{$\tau$-measurable} if there exists $s>0$ such that $\tau(e^{|x|}(s,\infty))<\infty.$
The set of all $\tau$-measurable operators is denoted by $S(\tau)$. It is a $\ast$-algebra w.r.t. the strong sum, strong multiplication (denoted simply by $x+y$ and $xy$, respectively, for $x,y\in S(\tau)$). Denote by $S_{sa}(\tau)$ (respectively, $S_+(\tau)$) the set of all self-adjoint (respectively, positive) elements of $S(\tau)$. General facts about $\tau$-measurable operators may be found in \cite{DPS2023}, \cite{Nelson1974}, \cite{Segal1953}, \cite{Terp1981} or \cite[Section IX.2]{T2}.

For every $x\in S(\tau)$, the \emph{distribution function} of $|x|$ is defined by setting
$$d(s;|x|)=\tau(e^{|x|}(s,\infty)),\quad s\geq 0.$$

\begin{definition} Let $x\in S(\tau).$ The \emph{singular value function} 
$\mu(x):t\mapsto \mu(t;x)$ of the operator $x$, is defined by
$$\mu(t;x)=\inf\{s\geq0: d(s;|x|)\leq t\},\quad t\geq 0.$$
\end{definition}

We have (see \cite[Proposition 3.2.5]{DPS2023})
$$\mu(t;x)=\inf\{\|xp\|_{\infty}:\ p\in \mathcal{P}(\M),\ p(H)\subset {\rm dom}(x),\ \tau(\mathbf{1}-p)\leq t\},\quad t\geq0.$$
We refer the reader to \cite{DPS2023,Fack1986,LSZ2013} for more properties of singular value functions.


Suppose $\{x_i\}_{i\in I}$ is an increasing net in $S_{sa}(\tau)$ and suppose $x\in S_{sa}(\tau).$ We write $x_i\uparrow x$ if $x$ is the least upper bound of $\{x_i\}_{i\in I}$ in $S_{sa}(\tau).$ 



The trace $\tau$ on $\M_+$ can be extended to $S_+(\tau)$ by setting (see \cite[Section 3.3]{DPS2023})
\begin{equation}\label{trace as integral}
\tau(x)=\int_0^\infty \mu(t;x)dt,\quad x\in S_+(\tau).
\end{equation}
Moreover, it can be extended to the set 
$$\{x\in S_{sa}(\tau):x_+\in L_1(\M) \text{ or } x_-\in L_1(\M)\}$$ by the linearity.

The following lemma can be found in Proposition 3.3.3 (iv) of \cite{DPS2023}.

\begin{lemma}\label{DDP1993 3.9}
If a net $0\le x_i\uparrow x$ in $S(\tau)$ then $\tau(x_i)\uparrow \tau(x).$
\end{lemma}

The following lemma can be found in Propositions 3.4.30 and 3.4.32 in \cite{DPS2023}.
\begin{lemma}\label{DDP1993 3.4}
If $x,y\in S(\tau)$ and if $xy,yx\in L_1(\M),$ then $\tau(xy)=\tau(yx).$ If in addition, $0\le x\in S(\tau)$ and $0\le y\in S(\tau)$ then $x^{\frac12}yx^{\frac12},y^{\frac12}xy^{\frac12}\in L_1(\M)$ and 
$$\tau(xy)=\tau(x^{\frac12}yx^{\frac12})=\tau(y^{\frac12}xy^{\frac12}).$$
\end{lemma}

The following lemma follows from Lemmas \ref{DDP1993 3.9} and \ref{DDP1993 3.4} immediately.
\begin{lemma}\label{trace limit lem} Let $0\le x\in L_1(\M)+\M$. If $\{r_i\}_{i\in I}$ is a net in $\mathcal{F}_1^+(\M)$ such that $r_i\uparrow {\bf 1}$, then
$$\tau(x)=\lim_{i} \tau(r_ix)$$
\end{lemma}


For any self-adjoint operators $a,b\in B(H)$ we say that $a$ is \emph{orthogonal} to $b$ if $ab=0.$ Note that this is equivalent to say that $\mathfrak{l}(a)\mathfrak{l}(b)=0.$

We need the following simple lemma. The proof is standard and thus omitted. 
\begin{lemma}\label{trivial lemma}
Let $\{a_m\}_{m\ge 0}$ be a sequence of pairwise orthogonal self-adjoint operators in $B(H)$. If there exists a constant $M>0$ such that $\|a_m\|_{\infty}\le M$ for each $m\ge0$, then 
$\sum_{m\ge0}a_m$ converges in the strong operator topology.
\end{lemma}

\subsection{Symmetric spaces}\label{symm subsec} 
We now recall the definition of symmetric spaces associated with a semifinite von Neumann algebra, which is an analogue of Banach ideals in $B(H)$. For these notions we refer the reader to \cite{DDP1993, DPS2023, KS2008, Krein1982, LSZ2013}.

Let $S(0,\infty)$ be the set of all real-valued Lebesgue measurable functions $f$ on $(0,\infty)$ with $d(s;|f|):=\rho(\{x:|f(x)|>s\})<\infty$ for some $s>0,$ where $\rho(\cdot)$ is the Lebesgue measure.

\begin{definition}\label{symmetric function space def}
A \emph{symmetric function space} $(E,\|\cdot\|_E)$ is a Banach space of real-valued Lebesgue measurable functions on $(0,\infty)$ such that for every $y\in E$ and $x\in S(0,\infty)$ such that $\mu(x)\leq\mu(y)$ we have $x\in E$ and $\|x\|_E\leq \|y\|_E.$
\end{definition}

Suppose $E$ is a symmetric function space on $(0,\infty)$. Define operator space 
\begin{equation}\label{symetric space from function}
E(\M)=\{x \in S(\tau):\ \mu(x)\in E\},
\end{equation}
and set
\begin{equation}\label{E norm def}
\left\|x \right\|_{E(\mathcal{M})}=\left\|\mu(x )\right\|_E.
\end{equation}
From \cite{KS2008} (see also \cite[Theorem 3.5.5]{LSZ2013}), $(E(\M),\|\cdot\|_{E(\M)})$ is a Banach space and is called a {\em symmetric space}. In particular, for $1\le p< \infty$, if $E=L_p$ is the standard Lebesgue $L_p$ function space on $(0,\infty)$, we obtain the classical noncommutative $L_p$ spaces $L_p(\M)$, see e.g. \cite{DPS2023,
PX2003}. 
For convenience, we set $L_\infty(\M)=\M$ equipped with the uniform norm $\|\cdot\|_\infty.$

For a given pair of symmetric function spaces $E,F$ on $(0,\infty)$, the norm in $E(\M)\cap F(\M)$ (also denoted by $(E\cap F)(\M)$) is defined by setting
$$\|x\|_{E(\M)\cap F(\M)}=\max\{\|x\|_{E(\M)},\|x\|_{F(\M)}\},\quad x\in E(\M)\cap F(\M).$$
The norm in $E(\M)+F(\M)$ (also denoted by $(E+F)(\M)$) is defined by setting
$$\|x\|_{E(\M)+F(\M)}=\inf\{\|y\|_{E(\M)}+\|z\|_{F(\M)}:x=y+z,y\in E(\M),z\in F(\M)\}.$$

It is a well-known fact that (see e.g. \cite[Theorem II.4.1]{Krein1982})
$$L_1\cap L_\infty\subset E\subset L_1+L_\infty$$
with continuous embeddings. This can be extended to symmetric spaces associated with the von Neumann algebra $\M$, namely, for any symmetric function space $E$ on $(0,\infty)$, we have
$$(L_1\cap L_\infty)(\M)\subset E(\M)\subset (L_1+L_\infty)(\M)$$
with continuous embedding (see \cite[p. 395]{DPS2023}).

Lorentz spaces form an important class of symmetric function spaces. Details of Lorentz spaces can be found in e.g. \cite{Bennett1988, Krein1982}. Let $\psi$ be an increasing concave function on $[0,\infty)$ such that $\psi(0)=0$. The \emph{Lorentz space} $\Lambda_{\psi}$ associated with $\psi$ is the set of all Lebesgue measurable functions $f$ on $(0,\infty)$ such that 
$$\|f\|_{\Lambda_\psi}=\int_0^\infty\mu(t;f)d\psi(t)<\infty.$$ 

Applying \eqref{symetric space from function} to $E=\Lambda_{\psi}$, we obtain
\begin{equation}\label{Lorentz space from function}
\Lambda_{\psi}(\M)=\{a\in S(\tau):\int_0^\infty \mu(t;a)d\psi(t)<\infty\},
\end{equation}
equipped with the norm
$$\|a\|_{\Lambda_{\psi}(\M)}=\int_0^\infty \mu(t;a)d\psi(t),\quad a\in \Lambda_{\psi}(\M).$$

In the case when $\psi(t)=t^{\frac1p},$ $1\le p<\infty,$ we denote the space $\Lambda_{\psi}$ by $L_{p,1}.$ 

The following elementary fact can be found in \cite[Lemma I.3.3]{Krein1982}.

\begin{fact}\label{symm embedding lem} If $E,F$ are two symmetric function spaces on $(0,\infty)$ such that $F\subset E$, then $F$ is embedded in $E$ continuously.
\end{fact}



Recall that for a symmetric function space $E$ on $(0,\infty)$, we denote the fundamental function of $E$ by $\varphi_E$, i.e. $\varphi_E(t)=\|\chi_{(0,t)}\|_E,t\ge0.$ Let $E^0$ be the closure of $L_1\cap L_{\infty}$ in $E.$ It follows from \cite[p.225]{dodds2014normed} that
$$E^0(\M)=\overline{\mathcal{F}(\M)}^{E(\M)}=\overline{(L_1\cap L_\infty)(\M)}^{E(\M)}.$$ 


The following lemma is well-known (see e.g. \cite[Theorem II.4.8]{Krein1982}).

\begin{lemma}\label{separable lemma}
Let $E$ be a symmetric function space on $(0,\infty).$ $E$ is separable if and only if $E=E^0$ and $\varphi_E(0+)=0.$ 
\end{lemma}

The following lemma can be found in \cite[Theorem II.4.7]{Krein1982}.
\begin{lemma}\label{decreasing lemma}
For every symmetric function space $E$ on $(0,\infty),$ the function $t\mapsto \varphi_E(t)/t, t>0$ is decreasing.
\end{lemma}



The following assertion can be found in Lemma 4 (c) of \cite{SS}.

\begin{lemma}\label{lemma in SS}
Let $E$ be a symmetric function space on $(0,\infty)$. We have $E\not\subset L_1$ if and only if $\varphi_E(t)=o(t)$ as $t\to\infty.$
\end{lemma}

\subsection{Quasicentral modulus}\label{qm subsec}

We now recall a theorem in \cite{BSZZ2024} concerning the characterization of the diagonality of an $n$-tuple in $\M$ modulo a given symmetric space, using quasicentral modulus. For $n$-tuples $\alpha,\beta\in \M^n$, $a\in\M$ and $\lambda\in \mathbb{C}$, we write $$[a,\alpha]=([a,\alpha(j)])_{j=1}^n,\quad \alpha+\lambda \beta=(\alpha(j)+\lambda\beta(j))_{j=1}^n.$$

For $a,b\in S(\tau)$, the commutator $[a,b]=ab-ba.$ Suppose $E$ is a symmetric function space on $(0,\infty)$. For every $n$-tuple $\beta=(\beta(j))_{j=1}^n\in (E(\M))^n$, we write
$$\|\beta\|_{E(\M)}=\max_{1\le j\le n}\|\beta(j)\|_{E(\M)}.$$
For every $n$-tuple $\alpha\in \M^n$ and every $a\in\M,$ we write $[a,\alpha]=([a,\alpha(j)])_{j=1}^n.$
\begin{definition}\label{def_qc}
Let $E$ be a symmetric function space on $(0,\infty)$. For every $n$-tuple $\alpha\in \M^n$, we define the \emph{quasicentral modulus} $k_{E(\M)}(\alpha)$ of $\alpha$ by the formula (see Definition A.1 in \cite{BSZZ2023})
$$k_{E(\M)}(\alpha)= \sup_{a\in \mathcal{F}_1^+(\M)}\inf_{\stackrel{r\ge a}{r\in \mathcal{F}_1^+(\M)}}\|[r,\alpha]\|_{E(\M)}.$$
Here $\mathcal{F}_1^+(\M)=\{x\in\M:0\le x\le \mathbf{1},\tau(\mathfrak{l}(x))<\infty\}$, where $\mathfrak{l}(x)$ is the left support of $x$ (see Subsection \ref{measur subsec}).
\end{definition}

\begin{definition}\label{approx unit} A sequence $\{r_k\}_{k\in \mathbb{N}}$ in $\mathcal{F}_1^+(\M)$ is called a {\em quasicentral approximate unit for $\alpha$ relative to $E(\M)$} if $r_k\to {\bf 1}$ in the strong operator topology and $\|[r_k,\alpha]\|_{E(\M)}\to0$ as $k\to\infty.$ In particular, if $E=L_{\infty}$, we simply call $\{r_k\}_{k\in \mathbb{N}}$ a {\em quasicentral approximate unit for $\alpha$}.
\end{definition}

From \cite[Theorem 3.6]{BSZZ2023}, $k_{E(\M)}(\alpha)=0$ if and only if there exists a quasicentral approximate unit for $\alpha$ relative to $E(\M)$.

The following theorem was proved in \cite[Theorem 1.2]{BSZZ2024} (together with Remark 1.3 there).

\begin{theorem}\label{thm_1.2_BSZZ2024}
Let $\M$ be a $\sigma$-finite von Neumann algebra with a faithful normal semifinite trace $\tau$. Let $E$ be a symmetric function space on $(0,\infty)$. For every commuting self-adjoint $n$-tuple $\alpha\in \M^n$, the following statements are equivalent
\begin{enumerate}[{\rm (i)}]
\item $k_{E(\M)}(\alpha)=0$;
\item For every $\e>0,$ there exists a diagonal $n$-tuple $\delta\in \M^n$ such that
$$\alpha-\delta\in (E^0(\M))^n,\quad \|\alpha-\delta\|_{(E\cap L_\infty)(\M)}<\e.$$
\end{enumerate}
\end{theorem}

\section{Generating projections}\label{construction sec}

In this section we introduce an important ingredient in the proof of Theorem \ref{thm_main} --- generating projections for a given $n$-tuple $\alpha\in ({\rm Aff(\M)})^n$. It was introduced in \cite[Definition 6.1]{PSW2002} (for the case $n=1$ and the special case of semifinite von Neumann algebra equipped with a faithful normal semifinite trace).
\begin{definition}\label{generating proj def} Let $\M$ be a von Neumann algebra and $n\in \mathbb{N}$. A projection $q\in\M$ is called a {\em generating projection} for a commuting self-adjoint $n$-tuple $\alpha\in({\rm Aff}(\M))^n$ if
$$\bigvee_{B\in \mathfrak{B}(\mathbb{R}^n)}\mathfrak{l}(e^{\alpha}(B)q)=\mathbf{1},$$
where $\mathfrak{B}(\mathbb R^n)$ is the collection of all Borel subsets in $\mathbb{R}^n.$
\end{definition}

The proof of the following lemma is straightforward and thus omitted.
\begin{lemma}\label{left support trivial lemma} Let $\M$ be a von Neumann algebra and let $p_1,p_2,q\in \M$ be projections. If $p_1$ is orthogonal to $p_2,$ then
\begin{equation} \mathfrak{l}((p_1+p_2)q)\le \mathfrak{l}(p_1q)+\mathfrak{l}(p_2q).
\end{equation}
\end{lemma}

We say that a subset $Q\subset \mathbb{R}^n$ is a \emph{box} if 
$$Q=I_1\times I_2\times\cdots\times I_n,$$
where each $I_j,$ $1\leq j\leq n,$ is a (closed, open or half-closed) real interval.

\begin{lemma}\label{lemma box}
If $q\in\M$ is a projection, then 
$$\bigvee\{\mathfrak{l}(e^{\alpha}(B)q):\ B\in{\rm Box}(\mathbb{R}^n) \}=\bigvee\{\mathfrak{l}(e^{\alpha}(B)q):\ B\in \mathfrak{B}(\mathbb{R}^n)\},$$
where ${\rm Box}(\mathbb{R}^n)$ is the collection of all boxes in $\mathbb{R}^n.$
\end{lemma}
\begin{proof} 
Denote the left-hand side by $p.$  Let 
$$\mathcal{A}=\{B\in \mathfrak{B}(\mathbb{R}^n):\ pe^{\alpha}(B)q=e^{\alpha}(B)q\}.$$

If $B$ is a box in $\mathbb{R}^n$, then $B\in\mathcal{A}.$ If $B$ is an elementary set, i.e. $B=\cup_{k=1}^mB_k,$ where $\{B_k\}_{k=1}^m$ are pairwise disjoint boxes, then
$$e^{\alpha}(B)=\sum_{k=1}^me^{\alpha}(B_k),$$
$$pe^{\alpha}(B)q=\sum_{k=1}^mpe^{\alpha}(B_k)q=\sum_{k=1}^me^{\alpha}(B_k)q=e^{\alpha}(B)q.$$
Thus, $B\in\mathcal{A}.$ Consequently, every elementary set is in $\mathcal{A}.$
	
Now assume that $\{B_m\}_{m\geq0}\subset\mathcal{A}$ is an increasing (respectively, decreasing) sequence and let $B=\cup_{m\geq0}B_m$ (respectively, $B=\cap_{m\geq0}B_m$). Then $e^{\alpha}(B_m)\to  e^{\alpha}(B)$ in the strong operator topology as $m\to\infty$. Hence,
$$pe^{\alpha}(B)q=\lim_{m\to\infty}pe^{\alpha}(B_m)q=\lim_{m\to\infty}e^{\alpha}(B_m)q=e^{\alpha}(B)q.$$
In either case, $B\in\mathcal{A}.$ This shows that $\mathcal{A}$ is a monotone class (see definition in p.65 of \cite{Folland}).

Thus, $\mathcal{A}$ is a monotone class containing every elementary set.
By Monotone Class Lemma (see e.g. Lemma 2.35 in \cite{Folland}), $\mathcal{A}$ coincides with $\mathfrak{B}(\mathbb{R}^n)$. In other words, we have
$$pe^{\alpha}(B)q=e^{\alpha}(B)q,\quad B\in\mathfrak{B}(\mathbb{R}^n).$$
Equivalently,
$$\mathfrak{l}(e^{\alpha}(B)q)\leq p,\quad B\in\mathfrak{B}(\mathbb{R}^n).$$
Thus,
$$\bigvee\{\mathfrak{l}(e^{\alpha}(B)q):\ B\in \mathfrak{B}(\mathbb{R}^n)\}\leq p.$$
The converse inequality is obvious.
\end{proof}

For every $m\in \mathbb{Z}_+$, let ${\rm At}_m$ be the collection of all cubes
$$\Big[\frac{k_1}{2^m},\frac{k_1+1}{2^m}\Big)\times\cdots\times\Big[\frac{k_n}{2^m},\frac{k_n+1}{2^m}\Big),\quad (k_1,\ldots,k_n)\in \mathbb{Z}^n.$$

\begin{lemma}\label{corrected PSW lemma} Let $\alpha\in ({\rm Aff}(\M))^n$ be a commuting self-adjoint $n$-tuple. If $q$ is a generating projection for $\alpha$, then
$$\bigvee_{m\geq0}\bigvee_{B\in {\rm At}_m}\mathfrak{l}(e^{\alpha}(B)q)=\mathbf{1}.$$
\end{lemma}
\begin{proof} Denote the projection on the left-hand side by $p.$ We have $(\mathbf{1}-p)\perp \mathfrak{l}(e^{\alpha}(B)q)$ for every $B\in {\rm At}_m$ for every $m\geq0.$ In other words,
\begin{equation}\label{perp to cubes}
(\mathbf{1}-p)e^{\alpha}(B)q=0,\quad B\in {\rm At}_m,\quad m\geq0.
\end{equation}

Let $\mathbf{F}_m$ be the $\sigma$-algebra generated by ${\rm At}_m.$ Note that $\mathbf{F}_m$ is nothing but the collection of all countable unions of pairwise disjoint elements in ${\rm At}_m.$ Thus, by \eqref{perp to cubes} and the $\sigma$-additivity of the spectral measure, we have
$$(\mathbf{1}-p)e^{\alpha}(B)q=0,\quad B\in\mathbf{F}_m,\quad m\geq0.$$
In other words,
$$(\mathbf{1}-p)e^{\alpha}(B)q=0,\quad B\in\bigcup_{m\geq0}\mathbf{F}_m.$$

Take now an arbitrary open box $B$ and select a sequence $\{B_k\}_{k\geq0}\subset \bigcup_{m\geq0}\mathbf{F}_m$ such that $B_k\uparrow B.$ We have
$$(\mathbf{1}-p)e^{\alpha}(B)q={\rm s.o.t}\text{-}\lim_{k\to\infty}(\mathbf{1}-p)e^{\alpha}(B_k)q=0,$$
where s.o.t is a shorthand for the strong operator topology.

Take now an arbitrary box $B$ and select a sequence $\{B_k\}_{k\geq0}$ of open boxes such that $B_k\downarrow B.$ We have
$$(\mathbf{1}-p)e^{\alpha}(B)q={\rm s.o.t}\text{-}\lim_{k\to\infty}(\mathbf{1}-p)e^{\alpha}(B_k)q=0.$$
In other words, $\mathbf{1}-p\perp \mathfrak{l}(e^{\alpha}(B)q)$ for every box $B.$ Thus,
$$(\mathbf{1}-p)\perp \bigvee\{\mathfrak{l}(e^{\alpha}(B)q):\ B\in{\rm Box}(\mathbb{R}^n) \}.$$
It follows from Lemma \ref{lemma box} that $$\bigvee\{\mathfrak{l}(e^{\alpha}(B)q):B\in{\rm Box}(\mathbb{R}^n) \}=\mathbf{1}.$$
Therefore, $p=\mathbf{1}.$
\end{proof}

In the lemma below, we prove that for any projection $q\in\M$ there is a projection $p\in\M,p\ge q,$ such that $p$ commutes with $\alpha$ and that $q$ is a generating projection of $\alpha p$ in the von Neumann algebra $p\M p.$ This together with the next lemma, allows us to reduce the proof of Theorem \ref{thm_main} to the case when $\alpha$ admits a generating $\tau$-finite projection.

\begin{lemma}\label{qe lemma} Let $\M$ be a von Neumann algebra and let $\alpha\in\M^n$ be a commuting self-adjoint $n$-tuple. For every projection $q\in\M,$ the projection
$$e_q=\bigvee_{B\in \mathfrak{B}(\mathbb{R}^n)}\mathfrak{l}(e^{\alpha}(B)q)$$
commutes with $\alpha,$ and $q$ is a generating projection for $e_q\alpha$ in $e_q\M e_q.$
\end{lemma}
\begin{proof} We claim that $e_q$ commutes with $\alpha.$ Indeed, from the definition of the supremum of a family of projections, we have (see e.g. \cite[p. 111]{Kadison1997})
$$e_q(H)=\overline{\sum_{B\in \mathfrak{B}(\mathbb{R}^n)}e^{\alpha}(B)q(H)}.$$
Here (and below in the proof) the symbol $\sum_{i\in I} X_i$, where $\{X_i\}_{i\in I}$ is a family of linear subsets in $H$,  denotes the set of all elements of the form $\sum_{i\in I} x_i$ where $x_i\in X_i$ and only finitely many $x_i$'s are non-zero.
 
Since $e^{\alpha}(\mathbb{R}^n)=\mathbf{1}$, it follows that $q(H)\subset e_q(H)$ and thus 
\begin{equation}\label{q domimated by e_q}
q\le e_q.
\end{equation}
For every $A\in \mathfrak{B}(\mathbb{R}^n),$ we have
$$e^{\alpha}(A)e_q(H)\subset \overline{\sum_{B\in \mathfrak{B}(\mathbb{R}^n)}e^{\alpha}(A)e^{\alpha}(B)q(H)}=\overline{\sum_{B\in\mathfrak{B}(\mathbb{R}^n)}e^{\alpha}(A\cap B)q(H)}\subset e_q(H).$$
Hence,
$$\mathfrak{l}(e^{\alpha}(A)e_q)\leq e_q.$$
In other words,
$$e^{\alpha}(A)e_q=e_qe^{\alpha}(A)e_q.$$
Taking adjoints, we obtain
$$e_qe^{\alpha}(A)=e_qe^{\alpha}(A)e_q.$$
Thus, $e_q$ commutes with $e^{\alpha}(A)$ for all $A\in \mathfrak{B}(\mathbb{R}^n).$ Hence, $e_q$ commutes with $\alpha.$

Set $\mathcal{N}=e_q\M e_q$ with the identity element ${\bf 1}_{\mathcal{N}}=e_q$. Denote by $e^{\alpha e_q}_{\mathcal{N}}$ the spectral measure of $\alpha e_q$ in $\mathcal{N}.$ It is not hard to verify that 
$$e^{\alpha e_q}_{\mathcal{N}}(B)=e^{\alpha}(B)e_q,\quad B\in \mathfrak{B}(\mathbb{R}^n).$$  
We have
$$\overline{\sum_{B\in \mathfrak{B}(\mathbb{R}^n)}e^{\alpha e_q}_{\mathcal{N}}(B)q(H)}=\overline{\sum_{B\in \mathfrak{B}(\mathbb{R}^n)}e^{\alpha}(B)e_qq(H)}\overset{\eqref{q domimated by e_q}}{=}\overline{\sum_{B\in \mathfrak{B}(\mathbb{R}^n)}e^{\alpha}(B)q(H)}=e_q(H).$$ 
This implies that
$$\bigvee_{B\in \mathfrak{B}(\mathbb{R}^n)}\mathfrak{l}(e_{\mathcal{N}}^{\alpha e_q}(B)q)={\bf 1}_{\mathcal{N}}.$$
Therefore, $q$ is a generating projection for $\alpha e_q$ in $\mathcal{N}$ (see Definition \ref{generating proj def}).
%
\end{proof}

\begin{lemma}\label{generating projections sequence lemma} Let $n\in\mathbb{N}$ and let $\M$ be a $\sigma$-finite von Neumann algebra with a normal semifinite faithful trace $\tau.$ Let $\alpha\in\M^n$ be a commuting self-adjoint $n$-tuple. Define the mapping $q\mapsto e_q$ on the lattice $\mathcal{P}(\M)$ as in Lemma \ref{qe lemma}. There exists a sequence $\{q_k\}_{k\ge1}$ of pairwise orthogonal $\tau$-finite projections in $\M$ such that the projections $\{e_{q_k}\}_{k\ge1}$ are pairwise orthogonal and $\sum_{k\ge 1} e_{q_k}=\mathbf{1}.$
\end{lemma}
\begin{proof} By Zorn's Lemma, there exists a maximal family of non-zero $\tau$-finite projections $\{q_i\}_{i\in \mathbb{I}}$ in $\M$ such that the projections $\{e_{q_i}\}_{i\in I}$ are pairwise orthogonal.

We claim that $\sum_{i\in \mathbb{I}} e_{q_i}=\mathbf{1}$. Otherwise, the projection $p=\mathbf{1}-\sum_{i\in \mathbb{I}} e_{q_i}\ne 0$. Choose a non-zero $\tau$-finite projection $f\leq p.$ Since each $e_{q_i}$ commutes with $\alpha$ (see Lemma \ref{qe lemma}), it follows that $p$ commutes with $\alpha.$ In particular, $pe^{\alpha}(B)p=e^{\alpha}(B)p$ for every $B\in\mathfrak{B}(\mathbb{R}^n).$ Multiplying on the right by $f,$ we obtain $pe^{\alpha}(B)f=e^{\alpha}(B)f$ for every $B\in \mathfrak{B}(\mathbb{R}^n).$ Thus, $\mathfrak{l}(e^{\alpha}(B)f)\leq p$ for every $B\in\mathfrak{B}(\mathbb{R}^n).$ Hence, $e_f\leq p.$ Then $e_f$ is orthogonal to $e_{q_i}$ for every $i\in \mathbb{I}$, which delivers a contradiction with the maximality of the family $\{q_i\}_{i\in \mathbb{I}}$. Hence, $\sum_{i\in \mathbb{I}} e_{q_i}=\mathbf{1}$.

Since $\M$ is $\sigma$-finite, it follows that $\mathbb{I}$ is countable. Thus, we get a sequence $\{q_k\}_{k\ge1}$ of $\tau$-finite projections in $\M$ such that the projections $\{e_{q_k}\}_{k\ge1}$ are pairwise orthogonal and $\sum_{k\ge 1} e_{q_k}=\mathbf{1}.$
\end{proof}

\section{Kuroda's theorem for $n$-tuples modulo Lorentz spaces in semifinite von Neumann algebras}\label{lorentz sec}

The following theorem is an important ingredient in the proof of Theorem \ref{thm_main}. It is a special case of Theorem \ref{thm_main} for Lorentz operator spaces $\Lambda_{\psi}(\M).$ Its proof is based on the construction from Section \ref{construction sec}. We prove it in the end of this section.

\begin{theorem}\label{kuroda thm lorentz}
Let $\mathcal{M}$ be a $\sigma$-finite von Neumann algebra with a faithful normal semifinite trace $\tau$. If $\Lambda_{\psi}\cap L_\infty\not\subset L_{n,1},$  then for every commuting self-adjoint $n$-tuple $\alpha\in\M^n$ and every $\e>0,$ there exists a diagonal $n$-tuple $\delta\in\M^n$ such that $$\|\alpha-\delta\|_{(\Lambda_{\psi}\cap L_\infty)(\M)}<\e.$$
\end{theorem}

Theorem \ref{kuroda thm lorentz} can be seen as a generalization of \cite[Proposition 3]{Bercovici1989}.

In the lemma below, under the assumption that $\alpha$ admits a generating $\tau$-finite projection, we construct a quasicentral approximate unit $\{p_m\}_{m\ge1}$ for $\alpha$ with good properties. In the case when $n=1$, $\{p_m\}_{m\ge1}$ satisfies $\|[p_m,\alpha]\|_{E(\M)}\to0$ as $m\to\infty$, whenever $E$ is a symmetric function space on $(0,\infty)$ such that $E\not\subset L_1$, see the proof Lemma \ref{wvn pre lemma}. This recovers Lemma 6.3 in \cite{PSW2002}.

\begin{lemma}\label{proj co}  Let $\mathcal{M}$ be a $\sigma$-finite von Neumann algebra with a faithful normal semifinite trace $\tau.$ Let $\alpha\in\M^n$ be a commuting self-adjoint $n$-tuple admitting a generating $\tau$-finite projection $q$. Suppose ${\rm Spec}(\alpha(j))\subset [0,1)$ for every $1\le j\le n$. There exists a sequence of projections $\{p_m\}_{m\ge0}$, $p_m\uparrow\mathbf{1}$, such that for every $m\ge0$ and every $1\le j\le n,$
\begin{enumerate}[\rm (i)]
\item\label{point 1} $\|[p_m,\alpha(j)]\|_{\infty}\le 2^{-m};$
\item\label{point 2} $\tau(p_m)\leq 2^{mn}\tau(q).$
\end{enumerate}
\end{lemma}
\begin{proof} Slightly abusing the notation, for every $m\in \mathbb{Z}_+$, we denote the collection of all cubes $$\Big[\frac{k_1}{2^m},\frac{k_1+1}{2^m}\Big)\times\cdots\times\Big[\frac{k_n}{2^m},\frac{k_n+1}{2^m}\Big),\quad (k_1,\ldots,k_n)\in \{0,1,\ldots,2^m-1\}^n,$$
by ${\rm At}_m.$ Define the operator
\begin{equation}\label{eq_pm}
p_m=\sum_{A\in{\rm At}_m}\mathfrak{l}(e^{\alpha}(A)q).
\end{equation}
Since the elements in ${\rm At}_m$ are pairwise disjoint, it follows that the range projections $\mathfrak{l}(e^{\alpha}(A)q)$ are pairwise orthogonal for different $A$'s in ${\rm At}_m$. Hence, each $p_m$ is a projection.

{\bf Step 1:} We claim that the sequence $\{p_m\}_{m\geq0}$ is increasing.

For every $A\in{\rm At}_m,$ we have
$$A=\bigcup_{\substack{B\in{\rm At}_{m+1}\\ B\subset A}}B,$$
where the cubes on the right-hand side are pairwise disjoint. It follows from Lemma \ref{left support trivial lemma} that
$$\mathfrak{l}(e^{\alpha}(A)q)\leq \sum_{\substack{B\in{\rm At}_{m+1}\\ B\subset A}}\mathfrak{l}(e^{\alpha}(B)q)\leq p_{m+1}.$$
Taking the supremum over $A\in{\rm At}_m,$ we arrive at $p_m\leq p_{m+1}.$

{\bf Step 2:} We claim that $p_m\uparrow \mathbf{1}.$

Note that
$$\bigvee_{m\geq0}\bigvee_{A\in{\rm At}_m}\mathfrak{l}(e^{\alpha}(A)q)=\bigvee_{m\geq0}p_m.$$
Since $q$ is a generating projection, it follows from Lemma \ref{corrected PSW lemma} that the left-hand side in the above equation is $\mathbf{1}.$ This yields Step 2.

{\bf Step 3:} For every $A\in{\rm At}_m,$ let $c_A\in A$ be the centre point of $A.$ We write $c_A=(c_A(j))_{j=1}^n$ where $c_A(j)\in \mathbb{R}.$ Set $n$-tuple $\alpha_m\in \M^n$ by
$$\alpha_m(j)=\sum_{A\in{\rm At}_m}c_A(j)e^{\alpha}(A),\quad 1\leq j\leq n.$$
We have
\begin{equation}\label{simple function approx}
\|\alpha(j)-\alpha_m(j)\|_{\infty}\le 2^{-m-1},\quad 1\le j\le n.
\end{equation}
We claim that $\alpha_m(j)$ commutes with $p_m$ for every $1\le j\le n.$

Indeed, for every $A\in{\rm At}_m,$ we have
$$p_m=\sum_{B\in{\rm At}_m}\mathfrak{l}(e^{\alpha}(B)q)=\mathfrak{l}(e^{\alpha}(A)q)+\sum_{\substack{B\in{\rm At}_m\\ B\neq A}}\mathfrak{l}(e^{\alpha}(B)q).$$
Thus,
$$\mathfrak{l}(e^{\alpha}(A)q)\leq p_m\leq \mathfrak{l}(e^{\alpha}(A)q)+\sum_{\substack{B\in{\rm At}_m\\ B\neq A}}e^{\alpha}(B)=\mathfrak{l}(e^{\alpha}(A)q)+e^{\alpha}([0,1)^n\backslash A),$$
where the last equality is because $\cup_{B\in {\rm At}_m} B=[0,1)^n.$
It follows immediately that $$e^{\alpha}(A)p_m=e^{\alpha}(A)\mathfrak{l}(e^{\alpha}(A)q)=\mathfrak{l}(e^{\alpha}(A)q).$$
Taking adjoints on both sides of the equality, we obtain that $p_m$ commutes with $e^{\alpha}(A)$ for every $A\in{\rm At}_m.$ This yields the claim.

{\bf Step 4:} By Step 3, $[p_m, \alpha_m(j)]=0$ for each $1\le j\le n.$ Hence,
$$[p_m,\alpha(j)]=[p_m,\alpha(j)-\alpha_m(j)],\quad 1\leq j\leq n.$$
Thus,
$$\|[p_m,\alpha(j)]\|_{\infty}\leq 2\|\alpha(j)-\alpha_m(j)\|_{\infty}\overset{\eqref{simple function approx}}{\le} 2^{-m},\quad 1\leq j\leq n.$$
This proves the assertion \eqref{point 1}. We also have
$$\tau(p_m)=\sum_{A\in{\rm At}_m}\tau(\mathfrak{l}(e^{\alpha}(A)q))=\sum_{A\in{\rm At}_m}\tau(\mathfrak{r}(e^{\alpha}(A)q))\leq\sum_{A\in{\rm At}_m}\tau(q)=2^{mn}\tau(q).$$
This proves the assertion \eqref{point 2}.
\end{proof}

Recall that we denote the fundamental function of $E$ by $\varphi_E,$ i.e. $\varphi_E(t)=\|\chi_{(0,t)}\|_E$ for all $t\ge0.$ The following lemma is useful.

\begin{lemma}\label{fundamental function trivial lemma} Let $E$ be a symmetric function space on $(0,\infty)$. For $\theta>0,$ we have 
$$\varphi_E(\theta t)\le \max\{\theta,1\}\varphi_E(t),\quad t\ge0.$$
\end{lemma}
\begin{proof}
If $\theta\le 1$, since $\varphi_E$ is an increasing function, it follows that $\varphi_E(\theta t)\le \varphi_E(t).$

Suppose $\theta>1.$ Noting that $t\mapsto \varphi_E(t)/t, t>0$ is a decreasing function (see Lemma \ref{decreasing lemma}), we have
$$\varphi_E(\theta t)\le \theta\varphi_E(t).$$
The assertion follows.
\end{proof}

The relation $\Lambda_{\psi}\cap L_\infty\not\subset L_{n,1}$ can be characterized in terms of the asymptotic behaviour of $\psi,$ as shown by the following lemma.

\begin{lemma}\label{conditions on psi} Let $n\in \mathbb{N}$ and let $\psi$ be an increasing concave function on $[0,\infty)$ such that $\psi(0)=0$. The following conditions are equivalent:
\begin{enumerate}[{\rm (i)}]
\item\label{cond 1} $\Lambda_{\psi}\cap L_{\infty}\not\subset L_{n,1};$
\item\label{cond 2} $\liminf_{t\to\infty}\frac{\psi(t)}{t^{\frac1n}}=0;$
\item\label{cond 3} $\liminf_{m\to\infty}\frac{\psi(2^{mn})}{2^m}=0.$
\end{enumerate}
\end{lemma}
\begin{proof} Suppose \eqref{cond 2} does not hold. Since $\liminf_{t\to\infty}\frac{\psi(t)}{t^{1/n}}>c$ for some $c>0,$ it follows that there exists large enough $s>0$ such that $\inf_{t\ge s} \frac{\psi(t)}{t^{1/n}}\ge c.$ In other words, $\psi(t)\ge ct^{1/n}$ for all $t\ge s.$ In particular, $\psi(k)\geq c'k^{\frac1n}$ for $k\in\mathbb{Z}_+,$ where $c'>0$ is a constant.

Let $x\in \Lambda_{\psi}\cap L_{\infty}.$ We have
\begin{multline*}
\|x\|_{L_{n,1}}\leq\Big\|\sum_{k=0}^\infty\mu(k;x)\chi_{[k,k+1)}\Big\|_{L_{n,1}}=\sum_{k=0}^\infty\mu(k;x)((k+1)^{\frac1n}-k^{\frac1n})\\
=\sum_{k=0}^\infty(\mu(k;x)-\mu(k+1;x))(k+1)^{\frac1n}\le \frac{1}{c'}\sum_{k=0}^\infty(\mu(k;x)-\mu(k+1;x))\psi(k+1)\\
=\frac{1}{c'}\sum_{k=0}^\infty\mu(k;x)(\psi(k+1)-\psi(k))\le \frac{1}{c'}\psi(1)\|x\|_\infty+\frac{1}{c'}\sum_{k=1}^\infty\mu(k;x)(\psi(k)-\psi(k-1))\\
\le \frac{1}{c'}\psi(1)\|x\|_{\infty}+\frac{1}{c'}\|x\|_{\Lambda_{\psi}}.
\end{multline*}
Therefore, $\Lambda_{\psi}\cap L_\infty\subset L_{n,1}.$ In other words, \eqref{cond 1} does not hold. Hence, \eqref{cond 1}$\to$\eqref{cond 2}.

Suppose \eqref{cond 1} does not hold. By Fact \ref{symm embedding lem}, there exists a constant $c\in (0,\infty)$ such that
$$\|x\|_{L_{n,1}}\le c\max\{\|x\|_{\Lambda_{\psi}},\|x\|_\infty\},\quad x\in \Lambda_{\psi}\cap L_\infty.$$
Applying this inequality to $x=\chi_{(0,t)}$, since $\int_0^\infty \chi_{(0,t)}d\psi(s)=\psi(t),$ we obtain that
$$t^{\frac1n}\le c\max\{\psi(t),1\},\quad t\in (0,\infty).$$
Thus, 
$$\liminf_{t\to\infty} \frac{\psi(t)}{t^{\frac1n}}\ge c^{-1}.$$
In other words, \eqref{cond 2} does not hold either. Thus, \eqref{cond 2}$\to$\eqref{cond 1}.

For $t\geq 1,$ set $m(t)=\lfloor\log_2(t^{\frac1n})\rfloor\in\mathbb{Z}_+.$ Here, $\lfloor x\rfloor$ denotes the integer part for $x\in \mathbb{R}_+.$ It is immediate that $$2^{m(t)}\leq t^{\frac1n}\leq 2^{m(t)+1},\quad t\geq 1.$$
Hence,
$$\frac{\psi(t)}{t^{\frac1n}}\leq\frac{\psi(2^{(m(t)+1)n})}{2^{m(t)}}\stackrel{{\rm Lemma}\ \ref{fundamental function trivial lemma}}{\le} 2^n\frac{\psi(2^{m(t)n})}{2^{m(t)}}\le 2^{n+1}\frac{\psi(t)}{t^{\frac1n}},\quad t\geq 1.$$
In particular,
$$\liminf_{t\to\infty}\frac{\psi(t)}{t^{\frac1n}}\leq 2^n\liminf_{t\to\infty}\frac{\psi(2^{m(t)n})}{2^{m(t)}}=2^n\liminf_{m\to\infty}\frac{\psi(2^{mn})}{2^{m}}
\le 2^{n+1}\liminf_{t\to\infty}\frac{\psi(t)}{t^{\frac1n}}.$$
Hence, \eqref{cond 2} is equivalent to \eqref{cond 3}.
\end{proof}

We need the following estimation of the norm $\|\cdot\|_{E(\M)}$ on $\mathcal{F}(\M)$. 
\begin{lemma}\label{symmetric norm estimation lemma} Let $\M$ be a von Neumann algebra with a faithful normal semifinite trace $\tau$. Let $E$ be a symmetric function space on $(0,\infty)$. For every $a\in \mathcal{F}(\M)$, we have
$$\|a\|_{E(\M)}\le \|a\|_\infty\varphi_E(\tau(\mathfrak{r}(a)))=\|a\|_\infty\varphi_E(\tau(\mathfrak{l}(a))).$$
\end{lemma}
\begin{proof}
Note that $\|a\|_{E(\M)}=\||a|\|_{E(\M)}$, and 
$$|a|\le \|a\|_\infty \mathfrak{r}(a).$$
So, we have $\|a\|_{E(\M)}\leq \|a\|_\infty \|\mathfrak{r}(a)\|_{E(\M)}$. Since $\mu(\mathfrak{r}(a))=\chi_{[0,\tau(\mathfrak{r}(a)))},$ by the definition of $\|\cdot\|_{E(\M)}$ we have
$$\|\mathfrak{r}(a)\|_{E(\M)} = \|\chi_{[0,\tau(\mathfrak{r}(a)))}\|_E=\varphi_E(\tau(\mathfrak{r}(a))).$$
Thus, 
$$\|a\|_{E(\M)}\leq \|a\|_\infty\varphi_E(\tau(\mathfrak{r}(a))).$$
Note that $\tau(\mathfrak{r}(a))=\tau(\mathfrak{l}(a)).$ The proof is complete.
\end{proof}

The following lemma yields Theorem \ref{kuroda thm lorentz} provided that there exists a generating $\tau$-finite projection.

\begin{lemma}\label{wvn for lorentz lemma} Suppose we are in the setting of Theorem \ref{kuroda thm lorentz}. Suppose in addition that $\alpha$ admits a generating $\tau$-finite projection $q$. For every $\e>0,$ there exists a diagonal $n$-tuple $\delta\in\M^n$ such that 
$$\|\alpha-\delta\|_{(\Lambda_{\psi}\cap L_\infty)(\M)}<\e.$$
\end{lemma}
\begin{proof} {\bf Step 1.} We prove the assertion assuming that $\Lambda_{\psi}\not\subset L_\infty.$ 

Simply adding $\|\alpha(j)\|_\infty\cdot\mathbf{1}$ to $\alpha(j)$ and rescaling, we may assume without loss of generality that ${\rm Spec}(\alpha(j))\subset [0,1)$ for every $1\le j\le n.$ From Lemma \ref{proj co}, there exists a sequence of $\tau$-finite projections $\{p_m\}_{m\ge0}$, $p_m\uparrow \mathbf{1}$ such that
$$\|[p_m,\alpha(j)]\|_{\infty}\le 2^{-m},\quad \tau(p_m)\leq 2^{mn}\tau(q),\quad m\geq0,\ 1\le j\le n.$$
Note that 
$$\mathfrak{l}([p_m,\alpha(j)])\le \mathfrak{l}(p_m\alpha(j))\vee \mathfrak{l}(\alpha(j)p_m).$$
Then 
$$\tau(\mathfrak{l}([p_m,\alpha(j)]))\le 2\tau(p_m)\le 2^{mn+1}\tau(q).$$
For every $m\ge0$, we have
\begin{multline*}
\|[p_m,\alpha(j)]\|_{\Lambda_{\psi}(\M)}\stackrel{{\rm Lemma}\ \ref{symmetric norm estimation lemma}}{\le} \|[p_m,\alpha(j)]\|_{\infty}\psi(\tau(\mathfrak{l}([p_m,\alpha(j)])))\\
\le  2^{-m}\psi(2^{mn+1}\tau(q)) \overset{{\rm Lemma}\ \ref{fundamental function trivial lemma}}{\le} \max\{2\tau(q),1\}\cdot 2^{-m}\psi(2^{mn}).
\end{multline*}
Thus, 
$$\|[p_m,\alpha(j)]\|_{\Lambda_{\psi}(\M)}\le \max\{2\tau(q),1\}\cdot 2^{-m}\psi(2^{mn}).$$
By Lemma \ref{conditions on psi}, there exists a strictly increasing sequence $\{m_k\}_{k\geq0}\subset\mathbb{Z}_+$ such that 
$$\max\{2\tau(q),1\}\cdot 2^{-m_k}\psi(2^{m_kn})\le 2^{-k},\quad k\geq0.$$
Then for every $1\le j\le n,$
$\|[p_{m_k},\alpha(j)]\|_{\Lambda_{\psi}(\M)}\to0$ as $k\to\infty.$ From \cite[Theorem 3.6]{BSZZ2023} (this is the only place in the proof where we use the assumption that $\Lambda_{\psi}\not\subset L_\infty$) we have
$$k_{\Lambda_{\psi}(\M)}(\alpha)=0.$$
Using Theorem \ref{thm_1.2_BSZZ2024}, we find a diagonal $n$-tuple $\delta\in\M^n$ such that
$$\|\alpha-\delta\|_{(\Lambda_{\psi}\cap L_\infty)(\M)}<\e.$$

{\bf Step 2.} Note that 
\begin{equation}\label{plus L1 eq}
(\Lambda_{\psi}+L_1)\cap L_\infty=\Lambda_{\psi}\cap L_\infty
\end{equation}
with equivalent norms. Thus, $(\Lambda_{\psi}+L_1)\cap L_{\infty}\not\subset L_{n,1}.$
Note that $\Lambda_{\psi}+L_1$ is also a Lorentz space, actually, $\Lambda_{\psi}+L_1=\Lambda_{\psi_1}$ where $\psi_1(t)=\min\{\psi(t),t\},t\ge0$ (see \cite[Theorem II.5.9]{Krein1982}). 

Since $L_1\subset \Lambda_{\psi_1}$, it follows that $\Lambda_{\psi_1}\not\subset L_\infty$. Thus, by Step 1,  for every $\e>0$, there exists a diagonal $n$-tuple $\delta\in\M^n$ such that 
$$\|\alpha-\delta\|_{(\Lambda_{\psi_1}\cap L_\infty)(\M)}\le \e. $$
Combining with \eqref{plus L1 eq}, the proof is complete.
\end{proof}

\begin{proof}[Proof of Theorem \ref{kuroda thm lorentz}] Let $q\mapsto e_q$ be the mapping on the lattice $\mathcal{P}(\M)$ defined as in Lemma \ref{qe lemma}. By Lemma \ref{generating projections sequence lemma}, there exists a sequence $\{q_k\}_{k\ge1}$ of pairwise orthogonal $\tau$-finite projections in $\M$ such that the projections $\{e_{q_k}\}_{k\ge1}$ are pairwise orthogonal and $\sum_{k\ge 1} e_{q_k}=\mathbf{1}.$ By Lemma \ref{qe lemma}, each $e_{q_k}$ commutes with $\alpha$ and $q_k$ is a generating function of $\alpha e_{q_k}$ in $e_{q_k}\M e_{q_k}.$ Thus, for every $\e>0,$ by Lemma \ref{wvn for lorentz lemma} (applied to the pair $(\alpha e_{q_k},e_{q_k}\M e_{q_k})$), there exists a diagonal $n$-tuple $\delta_k\in (e_{q_k}\M e_{q_k})^n$ such that
\begin{equation}\label{estimate lorentz subspace}
\|\alpha e_{q_k}-\delta_k\|_{(\Lambda_{\psi}\cap L_\infty)(e_{q_k}\M e_{q_k})}< 2^{-k}\e.
\end{equation}

Set $\delta(j)=\sum_{k\geq1}\delta_k(j),1\le j\le n.$ Since the sequence $\{\delta_k(j)\}_{k\geq1}$ consists of pairwise orthogonal operators and since
$$\sup_{k\geq1}\|\delta_k(j)\|_{\infty}\leq\sup_{k\geq1}\|\alpha(j)e_{q_k}\|_{\infty}+2^{-k}\e\leq \|\alpha(j)\|_{\infty}+\e,\quad 1\le j\le n,$$
it follows from Lemma \ref{trivial lemma} that $\delta\in\M^n.$ It is clear that $\delta$ is a commuting diagonal $n$-tuple.
	
For each $1\le j\le n,$ we write
$$\alpha(j)-\delta(j)=\sum_{k\geq1}(\alpha(j) e_{q_k}-\delta_k(j)),$$
where the series converges absolutely in $(\Lambda_{\psi}\cap L_\infty)(\M)$ due to \eqref{estimate lorentz subspace}. By \eqref{estimate lorentz subspace} and the triangular inequality, we have $\|\alpha-\delta\|_{(\Lambda_{\psi}\cap L_\infty)(\M)}<\e.$
\end{proof}

\section{Kuroda's theorem for $n$-tuples in semifinite von Neumann algebras}\label{symetric space sec}

In this section, we give the proof of Theorem \ref{thm_main}. The proof is based on Theorem \ref{kuroda thm lorentz} and ideas from \cite[Section 2]{V1990}.

We need some results related to K\"othe dual. We refer the reader to \cite{DDP1993} and \cite{DPS2023} for the K\"othe dual theory in the non-commutative setting. Let $E$ be a symmetric function space on $(0,\infty)$ and $E(\M)$ the symmetric space associated with $\M$. The {\em K\"othe dual} $(E(\M))^\times$ of $E(\M)$ is defined by setting
$$(E(\M))^{\times}=\{y\in S(\tau):\sup\{\tau(|xy|):x\in E(\M),\|x\|_{E(\M)}\le 1\}<\infty\}.$$
For $y\in (E(\M))^{\times},$ the norm $\|y\|_{(E(\M))^{\times}}$ is defined by setting 
\begin{equation}\label{kothe norm def}
\|y\|_{(E(\M))^{\times}}=\sup\{\tau(|xy|):x\in E(\M),\|x\|_{E(\M)}\le 1\}.
\end{equation}

For a symmetric function space $E$ on $(0,\infty)$, $(E^{\times},\|\cdot\|_{E^{\times}})$ is exactly the {\em associated space} of $E$ defined in \cite{Krein1982}. Namely, 
$$E^{\times}=\{g\in S(0,\infty):\int_0^\infty f(t)g(t)dt<\infty\ \text{for all }f\in E, \|f\|_E\le1\},$$
\begin{equation}\label{dual norm}
\|g\|_{E^\times}=\sup_{\|f\|_E\le1}\int_0^\infty f(t)g(t)dt<\infty,\quad g\in E^{\times}.
\end{equation}
Moreover, $(E^{\times},\|\cdot\|_{E^{\times}})$ is also a symmetric function space on $(0,\infty)$ (see \cite[p.104]{Krein1982}).

Let $\psi$ be an increasing concave function on $[0,\infty)$ such that $\psi(0)=0.$ The Marcinkiewicz space $(M_{\psi},\|\cdot\|_{M_\psi})$ associated with $\psi$ is defined by setting
\begin{equation}\label{M space}
M_{\psi}=\Big\{x\in S(0,\infty):\sup_{t>0}\frac{1}{\psi(t)}\int_0^t \mu(t;x)dt<\infty\Big\},
\end{equation}
$$\|x\|_{M_{\psi}}=\sup_{t>0}\frac{1}{\psi(t)}\int_0^t \mu(t;x)dt,\quad x\in M_{\psi}.$$
It is well-known that $M_{\psi}$ coincides with the K\"othe dual of $\Lambda_{\psi}$, i.e. $M_{\psi}=\Lambda_{\psi}^{\times}$ (see \cite[p.114]{Krein1982}).

On $(E(\M))^n$ we consider the norm $\|\zeta\|_{E(\M)}=\max_{1\le j\le n}\{\|\zeta(j)\|_{E(\M)}\},$ and on $(E^{\times}(\M))^n$ we consider the norm $\|\zeta\|_{\inf}=\sum_{j=1}^n \|\zeta(j)\|_{E^{\times}(\M)}.$ It follows from Lemma \ref{DDP1993 3.4} that $\tau(xy)=\tau(yx)$ for $x\in E(\M)$ and $y\in E^{\times}(\M)$. Consequently, a duality pairing may be defined by setting 
\begin{equation}\label{duality pairing}
\langle \alpha,\beta \rangle=\sum_{j=1}^n\tau(\alpha(j)\beta(j))=\sum_{j=1}^n\tau(\beta(j)\alpha(j))
\end{equation}
for $\alpha\in (E(\M))^n,\ \beta\in (E^{\times}(\M))^n.$ Denote by $(E(\M))_{sa}$ the self-adjoint elements in $E(\M)$. The following lemma can be found in \cite[Lemma 2.4]{MS2008}.

\begin{lemma}\label{dual space lemma} Let $\M$ be a von Neumann algebra with a faithful normal semifinite trace $\tau.$ Let $E$ be a separable symmetric function space on $(0,\infty).$ We have $$((E(\M))_{sa}^n,\|\cdot\|_{E(\M)})^*=((E^{\times}(\M))_{sa}^n,\|\cdot\|_{\inf})$$
via the duality pairing given by \eqref{duality pairing}. Here $(\cdot)^{\ast}$ denotes the Banach dual.
\end{lemma}

In this section, we will frequently use the notation
$$y(\alpha,\gamma)=i\sum_{j=1}^n [\alpha(j),\gamma(j)],\quad \alpha,\gamma\in S_{sa}(\tau)^n,$$
where $S_{sa}(\tau)$ is the set of all self-adjoint elements in $S(\tau),$ and $[\cdot,\cdot]$ is the commutator.
 
\begin{proposition}\label{commmutator estimation prop 2} Let $\M$ be a $\sigma$-finite von Neumann algebra with a normal semifinite faithful trace $\tau$. Let $\alpha\in \M^n$ be a commuting self-adjoint $n$-tuple and let $\gamma\in (M_{\psi}(\M))^n$ be a self-adjoint $n$-tuple. If $(y(\alpha,\gamma))_-\in L_1(\M),$ then 
$$|\tau(y(\alpha,\gamma))|\le k_{\Lambda_{\psi}(\M)}(\alpha)\sum_{j=1}^n\|\gamma(j)\|_{M_{\psi}(\M)}.$$
\end{proposition}
\begin{proof} If $k_{\Lambda_{\psi}(\M)}(\alpha)=\infty,$ then there is nothing to prove. We, therefore, assume $k_{\Lambda_{\psi}(\M)}(\alpha)<\infty.$ By \cite[Theorem 3.6]{BSZZ2023}, choose a sequence $\{r_m\}_{m\ge1}\subset\mathcal{F}_1^+(\M)$ such that $r_m\uparrow \mathbf{1}$ as $m\to\infty$, and 
$$\lim_{m\to\infty}\|[r_m,\alpha]\|_{\Lambda_{\psi}(\M)}=k_{\Lambda_{\psi}(\M)}(\alpha).$$
Here, we recall that $[r,\alpha]=([r,\alpha(j)])_{j=1}^n$ for $r\in \M,\alpha\in\M^n.$

Clearly, 
$$r_my(\alpha,\gamma)=i\sum_{j=1}^n r_m[\alpha(j),\gamma(j)]=i\sum_{j=1}^n [r_m,\alpha(j)]\gamma(j)+i\sum_{j=1}^n[\alpha(j),r_m\gamma(j)].$$
Since $\gamma(j)\in L_1(\M)+\M$ and $r_m\in L_1(\M)\cap \M$, it follows that $r_m\gamma(j)\in L_1(\M).$ By the assumption, we have $\alpha(j)\in\M,$ and appealing to Lemma \ref{DDP1993 3.4} we obtain $\tau([\alpha(j),r_m\gamma(j)])=0$ for every $1\le j\le n.$ Thus,
$$\tau(r_my(\alpha,\gamma))=i\sum_{j=1}^n \tau([r_m,\alpha(j)]\gamma(j))$$
and
\begin{equation}\label{c-s ineq}
|\tau(r_my(\alpha,\gamma))|\le \sum_{j=1}^n |\tau([r_m,\alpha(j)]\gamma(j))|
\overset{\eqref{kothe norm def}}{\le} \sum_{j=1}^{n}\|[r_m,\alpha(j)]\|_{\Lambda_{\psi}(\M)}\|\gamma(j)\|_{M_{\psi}(\M)}.
\end{equation}
Since $(y(\alpha,\gamma))_+\in L_1(\M)+\M$, it follows from Lemma \ref{trace limit lem} that
$$\tau((y(\alpha,\gamma))_+)=\lim_{m\to\infty} \tau(r_m(y(\alpha,\gamma))_+).$$
Similarly, 
$$\tau((y(\alpha,\gamma))_-)=\lim_{m\to\infty} \tau(r_m(y(\alpha,\gamma))_-).$$
Note that $(y(\alpha,\gamma))_-\in L_1(\M).$ We have
\begin{multline*}
\tau(y(\alpha,\gamma))=\tau((y(\alpha,\gamma))_+)-\tau((y(\alpha,\gamma))_-)\\
=\lim_{m\to\infty} \tau(r_m(y(\alpha,\gamma))_+)-\lim_{m\to\infty}\tau(r_m(y(\alpha,\gamma))_-)=\lim_{m\to\infty} \tau(r_my(\alpha,\gamma)).
\end{multline*} 
Passing $m\to\infty$ in \eqref{c-s ineq}, we complete the proof.
\end{proof}

\begin{lemma}\label{trace est lem} Let $y=y^\ast\in (L_1+L_\infty)(\M)$ and let $p\in \mathcal{P}(\M)$ be $\tau$-finite. We have
$$\tau(py)\le \int_0^{\tau(p)} \mu(s;y)ds.$$
\end{lemma}
\begin{proof} Let $x=pyp.$ Since $p$ is $\tau$-finite, it follows that $x\in L_1(\M)$. We have
\begin{equation}\label{trace estimate}
\tau(py)\overset{{\rm Lemma}\ \ref{DDP1993 3.4}}{=}\tau(pyp)\le \tau(|x|)\overset{\eqref{trace as integral}}{=}\int_0^{\infty}\mu(s;|x|)ds.
\end{equation}
For every $s>0,$ we have 
$$\tau(e^{|x|}(s,\infty))\le \tau(e^{|x|}(0,\infty))\le \tau(p).$$
Thus, $d(s;|x|)\le \tau(p)$ for all $s>0.$ Hence, $\mu(t;|x|)=0$ for all $t\ge \tau(p).$ Combining this observation with \eqref{trace estimate}, we obtain that
$$\tau(py)\le \int_0^{\tau(p)}\mu(s;|x|)ds.$$
Note that $\mu(x)\le \mu(y)$ (see \cite[Proposition 3.2.7 (vi)]{DPS2023}). The proof is complete.
\end{proof}

\begin{lemma}\label{negative part intergtable} Let $y=y^{\ast}\in (L_1+L_{\infty})(\M)$ and let $r_0$ be a $\tau$-finite projection. If
\begin{equation}\label{trace lower bound}
\tau(ry)\geq 0,\quad \mbox{for all } r\geq r_0,\ r\in \mathcal{F}_1^+(\M),
\end{equation}
then $y_-\in L_1(\M).$
\end{lemma}
\begin{proof} Let $q_0=e^{y}(-\infty,0].$ Choose a sequence $\{q_k\}_{k\in \mathbb{N}}$ of $\tau$-finite projections such that $q_k\uparrow q_0.$ We have
$$\tau((r_0\vee q_k)y)\geq0,\quad k\in\mathbb{N}.$$
Thus,
$$\tau((r_0\vee q_k-q_k)y)\geq -\tau(q_ky)=\tau(q_ky_-),\quad k\in\mathbb{N}.$$
Clearly,
$$\tau(r_0\vee q_k-q_k)\leq\tau(r_0),\quad k\in\mathbb{N}.$$
Thus, by Lemma \ref{trace est lem},
$$\tau((r_0\vee q_k-q_k)y)\leq\int_0^{\tau(r_0)}\mu(s;y)ds,\quad k\in\mathbb{N}.$$
Hence,
\begin{equation}\label{estimate on finite space}
\tau(q_ky_-)\leq \int_0^{\tau(r_0)}\mu(s;y)ds,\quad k\in\mathbb{N}.
\end{equation}
Since $y_-\in L_1(\M)+\M$, it follows from Lemma \ref{trace limit lem} (applied to the von Neumann algebra $q_0\M q_0$ and the sequence $\{q_k\}_{k\in \mathbb{N}}\subset \mathcal{F}_1^+(q_0\M q_0)$) that
$$\tau(y_-)=\lim_{k\to\infty} \tau(q_ky_-).$$
Passing $k\to\infty$ in \eqref{estimate on finite space}, we conclude that 
$$\tau(y_-)\leq \int_0^{\tau(r_0)}\mu(s;y)ds.$$
This completes the proof.	
\end{proof}

\begin{lemma}\label{from hb lemma} Let $X$ be a Banach space over the field $\mathbb{R}$ and let $K\subset X$ be a convex set such that $K\cap B_X(0,1)=\varnothing,$ where $B_X(0,1)$ denotes the open unit ball in $X$ centred at zero. There exists an $f\in X^{\ast}$ such that $f(\zeta)\geq 1$ for every $\zeta\in K.$  
\end{lemma}
\begin{proof} Let $L=B_X(0,1).$ Since $K$ is convex and since $L$ is open and convex, it follows from the Hahn-Banach separation theorem that there is a non-zero real-valued functional $f\in X^\ast$ such that 
$$\sup_{\eta\in L}f(\eta)\leq \inf_{\zeta\in K}f(\zeta).$$
Since $0\in L,$ it follows that $\sup_{\eta\in L}f(\eta)\geq 0.$ If $\sup_{\eta\in L}f(\eta)=0,$ then $f|_L=0$ (indeed, if $f(\eta)<0$ for some $\eta\in L$ then $f(-\eta)>0$ and thus $\sup_{\eta\in L}f(\eta)>0$). Hence, $f=0,$ which is a contradiction. Therefore, $\sup_{\eta\in L}f(\eta)>0.$ Assume without loss of generality that $\sup_{\eta\in L}f(\eta)=1.$ This completes the proof.
\end{proof}

\begin{proposition}\label{existence of commutator 2} Let $\M$ be a $\sigma$-finite von Neumann algebra with a faithful normal semifinite trace $\tau.$ Suppose $E$ is a separable symmetric  function space on $(0,\infty).$ Let $\alpha\in \M^n$ be a commuting self-adjoint $n$-tuple. Suppose $k_{E(\M)}(\alpha)>0.$ There exists a non-zero self-adjoint $n$-tuple $\gamma\in (E^{\times}(\M))^n$ such that
\begin{enumerate}[\rm (i)]
\item\label{item0} $(y(\alpha,\gamma))_-\in L_1(\M)$;
\item\label{item1} $\tau(y(\alpha,\gamma))\neq0.$
\end{enumerate}
\end{proposition}
\begin{proof} Assume without loss of generality that $k_{E(\M)}(\alpha)>1.$ By Definition \ref{def_qc}, there is $r_0\in\mathcal{F}_1^+(\M)$ (without loss of generality, $r_0$ is a $\tau$-finite projection) such that
\begin{equation}\label{eoc eq0}
\|[r,\alpha]\|_{E(\M)}\ge 1,\quad \forall r\in \mathcal{F}_1^+(\M),\quad r\ge r_0.
\end{equation}

The set $\{r\in \mathcal{F}_1^+(\M),r\ge r_0\}$ is convex. Hence, the set
$$K=\{i[r,\alpha]:\quad r\in \mathcal{F}_1^+(\M),\quad r\ge r_0\}$$
is a linear image of a convex set and is, therefore, convex. Since $\alpha\in\M^n$ and $r\in \mathcal{F}_1^+(\M),$ it follows that $[r,\alpha]\in (\mathcal{F}(\M))^n.$ In particular, $[r,\alpha]\in (E(\M))^n.$ This shows that $K$ is a convex subset of $(E(\M))_{sa}^n$ which does not intersect the open unit ball of $(E(\M))_{sa}^n.$ By Lemma \ref{from hb lemma}, there exists a real-valued $f\in ((E(\M))_{sa}^n)^{\ast}$ such that $f(\zeta)\geq 1$ for every $\zeta\in K.$

By Lemma \ref{dual space lemma},  $((E(\M))_{sa}^n,\|\cdot\|_{E(\M)})^\ast=((E^{\times}(\M))_{sa}^n,\|\cdot\|_{{\rm inf}})$ via the duality pairing given by \eqref{duality pairing}. Hence, there exists $\gamma\in (E^{\times}(\M))_{sa}^n$ such that 
$$f((x(j))_{j=1}^n)=\sum_{j=1}^n\tau(x(j)\gamma(j)),\quad (x(j))_{j=1}^n\in(E(\M))_{sa}^n.$$

We have $f(\zeta)\geq 1$ for all $\zeta\in K.$ By the definition of $K,$ we have
\begin{equation}\label{intermidiate ineq}
\tau\Big(\sum_{j=1}^{n}i[r,\alpha(j)]\gamma(j)\Big)=f(i[r,\alpha])\geq 1,\quad \forall r\in \mathcal{F}_1^+(\M),\quad r\ge r_0.
\end{equation}

Clearly, 
$$ry(\alpha,\gamma)=i\sum_{j=1}^n r[\alpha(j),\gamma(j)]=i\sum_{j=1}^n [r,\alpha(j)]\gamma(j)+i\sum_{j=1}^n[\alpha(j),r\gamma(j)].$$
Since $\gamma(j)\in L_1(\M)+\M$ and $r\in L_1(\M)\cap \M$, it follows that $r\gamma(j)\in L_1(\M).$ Noting that $\alpha(j)\in \M,$ by Lemma \ref{DDP1993 3.4} we have $\tau([\alpha(j),r\gamma(j)])=0$ for every $1\le j\le n.$ Thus,
$$\tau(ry(\alpha,\gamma))=\tau\Big(\sum_{j=1}^{n}i[r,\alpha(j)]\gamma(j)\Big).$$
Combining with \eqref{intermidiate ineq},
\begin{equation}\label{eoc eq1}
\tau(ry(\alpha,\gamma))\geq 1,\quad \forall r\in \mathcal{F}_1^+(\M),\quad r\ge r_0.
\end{equation}
Using Lemma \ref{negative part intergtable}, we obtain that $(y(\alpha,\gamma))_-\in L_1(\M).$ This proves \eqref{item0}.

Note that the set $\{r\in \mathcal{F}_1^+(\M):r\ge r_0\}$ is a directed set. Since $$(y(\alpha,\gamma))_+,(y(\alpha,\gamma))_-\in L_1(\M)+\M,$$
it follows from Lemma \ref{trace limit lem} that
$$\tau((y(\alpha,\gamma))_+)=\lim_{\substack{r\in \mathcal{F}_1^+(\M) \\ r\geq r_0}}\tau(r(y(\alpha,\gamma))_+),\quad \tau((y(\alpha,\gamma))_-)=\lim_{\substack{r\in \mathcal{F}_1^+(\M) \\ r\geq r_0}}\tau(r(y(\alpha,\gamma))_-).$$
Note that $(y(\alpha,\gamma))_-\in L_1(\M).$ We have
$$\tau(y(\alpha,\gamma))=\tau((y(\alpha,\gamma))_+)-\tau(y(\alpha,\gamma))_-)=\lim_{\substack{r\in \mathcal{F}_1^+(\M) \\ r\geq r_0}}\tau(ry(\alpha,\gamma))\stackrel{\eqref{eoc eq1}}{\geq}1.$$
This proves \eqref{item1} and, hence, completes the proof.
\end{proof}

The following result implies that if $\alpha$ is not diagonal modulo $E(\M),$ then there exists a Lorentz ideal $\Lambda_{\psi}$ containing $E$ such that $\alpha$ is not diagonal modulo $\Lambda_{\psi}(\M)$ (recall Theorem \ref{thm_1.2_BSZZ2024}). It is an analogue of \cite[Proposition 2.6]{V1990}.

\begin{proposition}\label{larger obstruction ideal} 
Let $\M$ be a $\sigma$-finite von Neumann algebra with a faithful normal semifinite trace $\tau.$ Suppose $(E,\|\cdot\|_E)$ is a separable symmetric function space on $(0,\infty).$ Let $\alpha\in \M^n$ be a commuting self-adjoint $n$-tuple. If $k_{E(\M)}(\alpha)>0,$ then there exists an increasing concave function $\psi$ on $[0,\infty)$ such that $E\subset \Lambda_{\psi}$ and $k_{\Lambda_{\psi}(\M)}(\alpha)>0.$
\end{proposition}
\begin{proof} Let $\gamma\in (E^\times(\M))^n$ be the non-zero $n$-tuple provided by Proposition \ref{existence of commutator 2}. Set $z=\sum_{j=1}^n|\gamma(j)|$ and set
\begin{equation}\label{formula of psi}
\psi(t)=\int_0^t \mu(s;z)ds,\quad t\in [0,\infty).
\end{equation}

For every $x\in L_1\cap L_\infty,$
$$\|\mu(x)\|_{\Lambda_{\psi}}=\int_0^\infty\mu(t;x)d\psi(t)=\int_0^\infty \mu(t;x)\mu(t;z)dt\overset{\eqref{dual norm}}{\le} \|\mu(x)\|_{E}\|\mu(z)\|_{E^{\times}}.$$
Thus, $E^0=\overline{L_1\cap L_\infty}^{\|\cdot\|_E}\subset \Lambda_{\psi}.$ In other words, $E\subset \Lambda_{\psi}$ since $E=E^0$ (by Lemma \ref{separable lemma}).

By \eqref{formula of psi} and the definition of $M_{\psi}$ (see \eqref{M space}), we have $z\in M_{\psi}(\M).$ In particular, $\gamma(j)\in M_{\psi}(\M)$ for $1\le j\le n.$ By Proposition \ref{existence of commutator 2} \eqref{item0}, we have $(y(\alpha,\gamma))_-\in L_1(\M).$ By Proposition \ref{existence of commutator 2} \eqref{item1}, we have $\tau(y(\alpha,\gamma))\neq0.$ It follows now from Proposition \ref{commmutator estimation prop 2} that $k_{\Lambda_{\psi}(\M)}(\alpha)>0.$ This completes the proof.
\end{proof}

\begin{lemma}\label{absolutely continuous lemma}
Let $(E,\|\cdot\|_E)$ be a symmetric function space on $(0,\infty)$. If $E^0\cap L_\infty\subset L_{n,1}$ then $E\cap L_\infty\subset L_{n,1}$.
\end{lemma}
\begin{proof} Suppose $E^0\cap L_\infty\subset L_{n,1}.$ By Fact \ref{symm embedding lem}, there exists a constant $c>0$ such that
$$\|x\|_{L_{n,1}}\le c\|x\|_{E\cap L_{\infty}},\quad x\in E^0\cap L_\infty.$$
Let $0\le x\in E\cap L_\infty$. For every $N\in \mathbb{N},$ note that $|x\chi_{(0,N]}|\le \|x\|_{\infty}\chi_{(0,N]}.$ We have $x\chi_{(0,N]}\in L_1\cap L_\infty\subset E^0.$ Thus, 
$$\|x\chi_{(0,N]}\|_{L_{n,1}}\le c\|x\chi_{(0,N]}\|_{E\cap L_{\infty}}\le c\|x\|_{E\cap L_{\infty}},\quad N\in \mathbb{N}.$$
Thus,
$$\sup_{N\in \mathbb{N}}\|x\chi_{(0,N]}\|_{L_{n,1}}\le c\|x\|_{E\cap L_\infty}.$$
It follows now from \cite[Proposition 9.2.1]{symmetric space book} that $x\in L_{n,1}$ and
$$\|x\|_{L_{n,1}}\le c\|x\|_{E\cap L_\infty}.$$
This completes the proof.
\end{proof}

The following result yields Theorem \ref{thm_main} in the case when $\alpha(j)$ is bounded for each $1\le j\le n.$

\begin{proposition}\label{second main result bounded case}
Let $\M$ be a $\sigma$-finite von Neumann algebra with a faithful normal semifinite trace $\tau.$ Suppose $(E,\|\cdot\|_E)$ is a symmetric function space on $(0,\infty)$. If $E\cap L_\infty\not\subset L_{n,1}$, then for every commuting self-adjoint $n$-tuple $\alpha\in\M^n$ and every $\e>0$, there exists a diagonal $n$-tuple $\delta\in\M^n$ such that 
$$\alpha-\delta\in (E^0(\M))^n,\quad \|\alpha-\delta\|_{(E\cap L_\infty)(\M)}<\e.$$ 
\end{proposition}
\begin{proof} Since $E\cap L_\infty\not\subset L_{n,1}$, it follows from Lemma \ref{absolutely continuous lemma} that $E^0\cap L_\infty\not\subset L_{n,1}.$ Hence, it suffices to prove the proposition for the case $E=E^0.$

Note that $E\cap L_\infty=(E+L_1)\cap L_\infty$ with equivalent norms. Thus, replacing $E$ with $E+L_1,$ we may assume that $L_1\subset E$ and thus $\varphi_E(0+)=0.$ Therefore, $E$ is separable by Lemma \ref{separable lemma}.

Suppose there is a commuting self-adjoint $n$-tuple $\alpha\in\M^n$ such that the assertion does not hold. It follows from Theorem \ref{thm_1.2_BSZZ2024} that $k_{E(\M)}(\alpha)>0.$ By Proposition \ref{larger obstruction ideal}, there exists an increasing concave function $\psi$ on $[0,\infty)$ such that $E\subset \Lambda_{\psi}$ and $k_{\Lambda_{\psi}(\M)}(\alpha)>0.$ Theorem \ref{kuroda thm lorentz} and Theorem \ref{thm_1.2_BSZZ2024} now yield $\Lambda_{\psi}\cap L_\infty\subset L_{n,1}.$ In particular, $E\cap L_\infty\subset L_{n,1}$, which is a contradiction. This proves the assertion.
\end{proof}

Now we are in the position to give the proof of Theorem \ref{thm_main}.

\begin{proof}[Proof of Theorem \ref{thm_main}]
For every $\mathbf{k}\in \mathbb{Z}^n$, denote $\mathbf{k}=(k_1,\ldots,k_n).$ Fix $\e>0.$
	
Let
$$p_{\mathbf{k}}=e^{\alpha}(\mathbf{k}+[0,1)^n),\quad \mathbf{k}\in\mathbb{Z}^n.$$
For every $\mathbf{k}\in \mathbb{Z}^n$, by Proposition \ref{second main result bounded case}, choose a commuting diagonal $n$-tuple $\delta_{\mathbf{k}}\in (p_{\mathbf{k}}\M p_{\mathbf{k}})^n$ such that
\begin{equation}\label{norm estimation 2}
\alpha\cdot p_{\mathbf{k}}-\delta_{\mathbf{k}}\in (E^0(p_{\mathbf{k}}\M p_{\mathbf{k}}))^n,\quad \|\alpha\cdot p_{\mathbf{k}}-\delta_{\mathbf{k}}\|_{(E\cap L_\infty)(p_{\mathbf{k}}\M p_{\mathbf{k}})}<\e\cdot 2^{-|\mathbf{k}|_1},
\end{equation}
where $|\mathbf{k}|_1=\sum_{i=1}^n|k_i|.$ Define $n$-tuple $\delta\in ({\rm Aff}(\M))^n$ by setting
\begin{equation}\label{diagonal tuple part}
\delta(j)\xi=\sum_{\mathbf{k}\in\mathbb{Z}^n}\delta_\mathbf{k}(j)\xi,\quad \xi\in {\rm dom}(\alpha(j)),\ 1\leq j\leq n.
\end{equation}
We now show that the series on the right-hand side converges in $H$. 

Fix $1\le j\le n.$ For every $\xi\in {\rm dom}(\alpha(j))$, noting that $\sum_{\mathbf{k}\in \mathbb{Z}^n}p_{\mathbf{k}}=\mathbf{1},$ we have
$$\alpha(j)\xi=\alpha(j)(\sum_{\mathbf{k}\in \mathbb{Z}^n}p_{\mathbf{k}}\xi)=\sum_{\mathbf{k}\in \mathbb{Z}^n}\alpha(j)p_{\mathbf{k}}\xi.$$
From \eqref{norm estimation 2}, $\sum_{\mathbf{k}\in \mathbb{Z}^n} \alpha(j)p_{\mathbf{k}}-\delta_{\mathbf{k}}(j)$ converges in the uniform norm and thus converges also in the strong operator topology. Therefore, $\sum_{\mathbf{k}\in \mathbb{Z}^n} \delta_{\mathbf{k}}(j)\xi$ converges for each $\xi\in {\rm dom}(\alpha(j)).$

Since $\{\delta_{\mathbf{k}}(j)\}_{\mathbf{k}\in \mathbb{Z}^n}$ are pairwise orthogonal diagonal operators, it follows that $\delta(j)$ is diagonal. Clearly, $\delta$ is a commuting $n$-tuple. By the triangle inequality and \eqref{norm estimation 2},
$$\|\alpha-\delta\|_{(E\cap L_\infty)(\M)}\le \sum_{\mathbf{k}\in\mathbb{Z}^n}\|\alpha\cdot p_{\mathbf{k}}-\delta_\mathbf{k}\|_{(E\cap L_\infty)(\M)}\leq\sum_{\mathbf{k}\in\mathbb{Z}^n}\e\cdot 2^{-|\mathbf{k}|_1}=3^n\e,$$
where the equality is due to $\sum_{k\in \mathbb{Z}}2^{-|k|}=3.$
Since $\alpha\cdot p_{{\bf k}}-\delta_{{\bf k}}\in (E^0(\M)\cap \M)^n$ for every ${\bf k}\in \mathbb{Z}^n$, it follows from the completeness of $(E^0(\M)\cap \M)^n$ that $\alpha-\delta\in (E^0(\M)\cap \M)^n.$
\end{proof}

\section{Final remarks}\label{sec_rem}

In this section, we present further remarks to the converse of Theorem \ref{thm_main} and Corollary \ref{cor_main}.

The converse of Corollary \ref{cor_main} holds for the case $\M=B(H_0)$, which is a consequence of the classical Kato-Rosenblum theorem. Indeed, Kato \cite{Kato1957} and Rosenblum \cite{Rosenblum1957} showed that, up to unitary equivalence, the absolutely continuous part of a self-adjoint operator in $B(H_0)$ is invariant under trace-class perturbations. Note that $L_1(B(H_0))$ is the trace class in $B(H_0)$. Thus, if the spectral measure of $a$ has non-zero absolutely continuous part, and if $E$ is a symmetric function space on $(0,\infty)$ satisfying $E\subset L_1$, then there does not exist diagonal operator $d\in B(H_0)$ such that $a-d\in E(B(H_0))$.

However, one should not expect a direct generalisation of the Kato-Rosenblum theorem to hold in the general semifinite von Neumann algebra $\M.$ 
Indeed, by \cite[Example 2.4.2]{LSSW2018}, there exists a semifinite von Neumann algebra $\M$ and a self-adjoint operator $a$ therein, such that the spectral measure of $a$ is absolutely continuous, but for every $\e>0,$ there is a diagonal $d\in\M$ satisfying $\|a-d\|_{L_1(\M)}<\e.$ 
By introducing the notion of {\em norm absolutely continuous projections,} Li et al. \cite{LSSW2018} obtained a generalisation of the Kato-Rosenblum theorem to semifinite von Neumann algebras. As a consequence, if $P_{ac}^\infty(a)\ne0,$ then there exists no diagonal $d\in\M$ such that $a-d\in L_1(\M).$ Here $P_{ac}^\infty(a)$ is the supremum of all norm absolutely continuous projections of $a$ (see \cite{LSSW2018}). The condition that $P_{ac}^\infty(a)\ne0$ can be characterised as follows: $P_{ac}^\infty(a)\ne0$ if and only if there is a self-adjoint element $c\in\M$ such that $i(ac-ca)\ge0$ (see \cite[Proposition 5.3.1]{LSSW2018}). The latter condition is equivalent to saying that $a$ is the real part of a non-normal hyponormal operator.


For the case $\M=B(H_0)$ and $n\ge 2$, the converse of Theorem \ref{thm_main} also holds. This is a consequence of the following extension of the Kato-Rosenblum theorem for $n$-tuples in $B(H_0)$, presented by Voiculescu in \cite[Theorem 2.2]{V1981} and \cite{V1979}. 

\begin{theorem*}[{\cite[Theorem 2.2]{V1981}}] If $\alpha,\beta\in (B(H_0))^n$ are commuting self-adjoint $n$-tuples such that $\alpha(j)-\beta(j)\in L_{n,1}(B(H_0))$ for every $1\le j\le n,$ then the absolutely continuous parts of $\alpha$ and $\beta$ are unitarily equivalent.
\end{theorem*}

By the theorem above, if a commuting self-adjoint $n$-tuple $\alpha\in (B(H_0))^n$ has non-zero absolutely continuous part, and if $E$ is a symmetric function space on $(0,\infty)$ satisfying $E\cap L_\infty\subset L_{n,1}$, there exists no commuting diagonal $n$-tuple $\delta\in (B(H_0))^n$ such that $\alpha(j)-\delta(j)\in E(B(H_0))$ for every $1\le j\le n$.

It remains an open problem to establish a proper extension of \cite[Theorem 2.2]{V1981} to the case of a general semifinite von Neumann algebra $\M$. 
To achieve this, it is unsurprising that an appropriate generalised notion of the ``absolutely continuous part'' of a commuting self-adjoint $n$-tuple $\alpha\in\M^n$ is required.


\appendix

%
%

\section{Kuroda's theorem for a single operator --- a proof without using quasicentral modulus}\label{app single operator}

In this section, we provide a simpler proof of Corollary \ref{cor_main}, which is a special case of Theorem \ref{thm_main} when $n=1$. Here, the quasicentral modulus is not needed.

The following lemma yields Corollary \ref{cor_main} for bounded operators provided that there exists a generating $\tau$-finite projection. 
\begin{lemma}\label{wvn pre lemma} Let $\mathcal{M}$ be a $\sigma$-finite von Neumann algebra with a faithful normal semifinite trace $\tau$ and let $a=a^\ast\in\mathcal{M}$ admit a generating $\tau$-finite projection $q$. Suppose $E$ is a symmetric function space on $(0,\infty)$ such that $E\not\subset L_1.$ 
For every $\varepsilon>0,$ there exists a diagonal operator $d\in\M$ such that 
$$a-d\in E^0(\M),\quad \|a-d\|_{(E\cap L_\infty)(\M)}<\e.$$
\end{lemma}
\begin{proof} Simply adding $\|a\|_\infty\cdot\mathbf{1}$ to $a$ and rescaling, we may assume without loss of generality that ${\rm Spec}(a)\subset [0,1)$. Let $\{p_m\}_{m\ge1}$ be an increasing sequence of $\tau$-finite projections constructed in Lemma \ref{proj co} (applied to $\alpha=a$). We have
$$\|[a,p_m]\|_{\infty}\leq 2^{-m},\quad \tau(p_m)\leq 2^m\tau(q),\quad m\geq0.$$
It is easy to see that $\mathfrak{l}([a,p_m])\le \mathfrak{l}(ap_m)\vee \mathfrak{l}(p_ma).$ Hence,
$$\tau(\mathfrak{l}([a,p_m]))\le 2\tau(p_m)\le 2^{m+1}\tau(q).$$
Thus, by Lemma \ref{symmetric norm estimation lemma},
$$\|[a,p_m]\|_{E(\M)}\leq 2^{-m}\varphi_E(2^{m+1}\tau(q)).$$
Combining with Lemma \ref{fundamental function trivial lemma}, we obtain
$$\|[a,p_m]\|_{E(\M)}\le \max\{1,2\tau(q)\}\cdot 2^{-m}\varphi_E(2^m),\quad m\geq0.$$

Since $E\not\subset L_1,$ by Lemma \ref{lemma in SS}, we have $2^{-m}\varphi_E(2^m)\to0$ as $m\to\infty.$ Thus, we can choose a strictly increasing sequence $\{m_k\}_{k\geq0}\subset\mathbb{Z}_+$ such that 
$$\max\{1,2\tau(q)\}\cdot 2^{-m_k}\varphi_E(2^{m_k})\leq 2^{-k}.$$ 
Set $q_k=p_{m_k}$ for $k\geq0.$ We have
$$\|[a,q_k]\|_{E(\M)}\leq 2^{-k},\quad \|[a,q_k]\|_{\infty}\leq 2^{-k},\quad k\geq0.$$
Denote for brevity $r_k=q_{k+1}-q_k,$ $k\geq0.$

For every $k\geq0,$ set
$$a_k=q_kaq_k+\sum_{l=k}^{\infty}r_lar_l.$$

{\bf Step 1:} We show that
$$\|a_k-a\|_{(E\cap L_\infty)(\M)}\leq 2^{2-k},\quad k\geq0.$$
Clearly,
$$a-a_k=q_ka(\mathbf{1}-q_k)+(\mathbf{1}-q_k)aq_k+(\mathbf{1}-q_k)a(\mathbf{1}-q_k)-\sum_{l\ge k} r_lar_l.$$
Noting that $\mathbf{1}-q_k=\sum_{l\ge k} r_l,$ we obtain 
\begin{multline*}
(\mathbf{1}-q_k)a(\mathbf{1}-q_k)-\sum_{l\ge k} r_lar_l=\sum_{i\ge k}\sum_{j\ge k} r_iar_j-\sum_{l\ge k} r_lar_l\\
=\sum_{l\ge k} \Big(\sum_{i\ge l+1} r_iar_l+r_la\sum_{j\ge l+1} r_j\Big)=\sum_{l\ge k}(\mathbf{1}-q_{l+1})ar_l+r_la(\mathbf{1}-q_{l+1}).
\end{multline*}
Hence,
$$a-a_k=q_ka(\mathbf{1}-q_k)+(\mathbf{1}-q_k)aq_k+\sum_{l\ge k}(\mathbf{1}-q_{l+1})ar_l+r_la(\mathbf{1}-q_{l+1}).$$
By the triangle inequality,
\begin{multline*}
\|a-a_k\|_{(E\cap L_\infty)(\M)}\leq 2\|q_ka(\mathbf{1}-q_k)\|_{(E\cap L_\infty)(\M)}+2\sum_{l\ge k}\|(\mathbf{1}-q_{l+1})ar_l\|_{(E\cap L_\infty)(\M)}\\
\leq 2\|q_k[a,q_k]\|_{(E\cap L_\infty)(\M)}+2\sum_{l\ge k}\|(\mathbf{1}-q_{l+1})aq_{l+1}\|_{(E\cap L_\infty)(\M)}\\
= 2\|q_k[a,q_k]\|_{(E\cap L_\infty)(\M)}+2\sum_{l\ge k}\|[a,q_{l+1}]q_{l+1}\|_{(E\cap L_\infty)(\M)}\\
\le 2\|[a,q_k]\|_{(E\cap L_\infty)(\M)}+2\sum_{l\ge k}\|[a,q_{l+1}]\|_{(E\cap L_\infty)(\M)}\leq 2\sum_{l\ge k}2^{-l}=2^{2-k}.
\end{multline*}

{\bf Step 2:} By the spectral theorem, every self-adjoint operator in a von Neumann algebra can be approximated in the uniform norm by diagonal operators. For every $l\geq 0,$ choose a diagonal operator $b_l\in r_l\M r_l$ such that
$$\|r_lar_l-b_l\|_{\infty}\leq\frac1{2^l\max\{1,\varphi_E(\tau(r_l))\}}.$$
It follows from Lemma \ref{symmetric norm estimation lemma} that
$$\|r_lar_l-b_l\|_{E(\M)}\le \|r_lar_l-b_l\|_{\infty}\varphi_E(\tau(r_l))\le 2^{-l}.$$

For every $k\geq0,$ choose a diagonal operator $c_k\in q_k\M q_k$ such that
$$\|q_kaq_k-c_k\|_{\infty}\leq\frac1{2^k\max\{1,\varphi_E(\tau(q_k))\}}.$$
It follows from Lemma \ref{symmetric norm estimation lemma} that
$$\|q_kaq_k-c_k\|_{E(\M)}\le \|q_kaq_k-c_k\|_{\infty}\varphi_E(\tau(q_k))\le 2^{-k}.$$

For every $k\ge0,$ set
$$d_k=c_k+\sum_{l\ge k}b_l,$$
where the series converges in the strong operator topology due to Lemma \ref{trivial lemma}. Since the summands are diagonal and pairwise orthogonal, it follows that $d_k$ is diagonal for every $k\geq0.$ By the triangle inequality,
\begin{multline*}
\|a_k-d_k\|_{(E\cap L_\infty)(\M)}\leq\|q_kaq_k-c_k\|_{(E\cap L_\infty)(\M)}+\sum_{l\ge k}\|r_lar_l-b_l\|_{(E\cap L_\infty)(\M)}\\
\leq 2^{-k}+\sum_{l\ge k}2^{-l}\leq 2^{2-k}.
\end{multline*}
Combining this with the estimate in Step 1, we obtain that
$$\|a-d_k\|_{(E\cap L_\infty)(\M)}\leq 2^{3-k},\quad k\geq0.$$
This completes the proof.
\end{proof}

The following lemma yields Corollary \ref{cor_main} in the case when $a$ is bounded.
\begin{lemma}\label{wvn pre bounded case} Let $\mathcal{M}$ be a $\sigma$-finite von Neumann algebra with a faithful normal semifinite trace $\tau$. Let $a\in \M$ be self-adjoint. Suppose $E$ is a symmetric function space on $(0,\infty)$ such that $E\not\subset L_1.$ For every $\varepsilon>0,$ there exists a diagonal operator $d\in \M$ such that $a-d\in E^{0}(\M)$ and $\|a-d\|_{(E\cap L_\infty)(\M)}<\e.$
\end{lemma}
\begin{proof} Let $q\mapsto e_q$ be the mapping on the lattice $\mathcal{P}(\M)$ defined as in Lemma \ref{qe lemma}. By Lemma \ref{generating projections sequence lemma}, there exists a sequence $\{q_k\}_{k\ge1}$ of pairwise orthogonal $\tau$-finite projections in $\M$ such that the projections $\{e_{q_k}\}_{k\ge1}$ are pairwise orthogonal and $\sum_{k\ge 1} e_{q_k}=\mathbf{1}.$ By Lemma \ref{qe lemma}, each $e_{q_k}$ commutes with $a,$ and $q_k$ is a generating projection for $ae_{q_k}$ in $e_{q_k}\M e_{q_k}.$ Thus, for every $\e>0,$ by Lemma \ref{wvn pre lemma} (applied to the pair $(ae_{q_k},e_{q_k}\M e_{q_k})$), there exists a diagonal element $d_k\in e_{q_k}\M e_{q_k}$ such that $ae_{q_k}-d_k\in E^0(e_{q_k}\M e_{q_k}),$
$$\|ae_{q_k}-d_k\|_{(E\cap L_\infty)(e_{q_k}\M e_{q_k})}< 2^{-k}\e.$$
In particular, $ae_{q_k}-d_k\in E^0(\M),$
\begin{equation}\label{estimate on subspace}
\|ae_{q_k}-d_k\|_{(E\cap L_\infty)(\M)}< 2^{-k}\e.
\end{equation}

Set $d=\sum_{k\geq1}d_k.$ Since the sequence $\{d_k\}_{k\geq1}$ consists of pairwise orthogonal elements and since
$$\sup_{k\geq1}\|d_k\|_{\infty}<\sup_{k\geq1}\|ae_{q_k}\|_{\infty}+2^{-k}\e\leq \|a\|_{\infty}+\e,$$
it follows from Lemma \ref{trivial lemma} that $d\in\M.$

We write
$$a-d=\sum_{k\geq1}(ae_{q_k}-d_k),$$
where the series converges absolutely in $(E\cap L_\infty)(\M)$ due to \eqref{estimate on subspace}, and we have $\|a-d\|_{(E\cap L_\infty)(\M)}< \e$. Since every summand in the right-hand side of the above equality belongs to $E^0(\M),$ it follows that $a-d\in E^0(\M).$
\end{proof}


\begin{proof}[Proof of Corollary \ref{cor_main}] 
Fix $\e>0.$ For every $k\in \mathbb{Z}$, let
$$p_{k}=e^{a}(k+[0,1)),\quad k\in\mathbb{Z}.$$
Applying Lemma \ref{wvn pre bounded case} to the pair $(ap_k,p_k\M p_k)$, we find a diagonal $d_k\in p_{k}\M p_{k}$ such that $ap_k-d_k\in E^0(p_k\M p_k)$ and
\begin{equation}\label{norm estimation}
\quad \|a p_{k}-d_{k}\|_{(E\cap L_\infty)(p_{\mathbf{k}}\M p_{\mathbf{k}})}<\e\cdot 2^{-|k|},\quad k\in\mathbb{Z}.
\end{equation}
Define operator $d\in {\rm Aff}(\M)$ by setting
\begin{equation}\label{diagonal part}
d\xi=\sum_{k\in\mathbb{Z}}d_k\xi,\quad \xi\in {\rm dom}(a).
\end{equation}
We now show that the series on the right-hand side converges in $H$. 

For every $\xi\in {\rm dom}(a)$, since $\sum_{k\in \mathbb{Z}}p_k=\mathbf{1},$ by the spectral theorem we have
$$a\xi=a(\sum_{k\in \mathbb{Z}}p_k\xi)=\sum_{k\in \mathbb{Z}}ap_k\xi.$$
From \eqref{norm estimation}, $\sum_{k\in \mathbb{Z}} ap_k-d_k$ converges in the uniform norm and thus converges also in the strong operator topology. Therefore, $\sum_{k\in \mathbb{Z}} d_k\xi$ converges for each $\xi\in {\rm dom}(a).$

Since $\{d_k\}_{k\in \mathbb{Z}}$ are pairwise orthogonal diagonal operators, it follows that $d$ is diagonal. By the triangle inequality,
$$\|a-d\|_{(E\cap L_\infty)(\M)}\leq\sum_{k\in\mathbb{Z}}\|a p_k-d_k\|_{(E\cap L_\infty)(\M)}< \sum_{k\in\mathbb{Z}}\e\cdot 2^{-|k|}=3\e,$$
and $a-d\in E^0(\M)\cap \M.$
\end{proof}


\section*{Acknowledgement}
F. Sukochev and D. Zanin are supported by the Australian Research Council DP230100434. H. Zhao acknowledges the support of Australian Government Research Training Program (RTP) Scholarship.


\end{document}